\documentclass[11pt,fleqn]{article}
\usepackage{amsfonts}
\usepackage{amssymb}
\usepackage{amsthm}
\usepackage{relsize}
\usepackage[leqno]{amsmath}
\usepackage{multicol}
\usepackage{blkarray}
\usepackage[export]{adjustbox}
\usepackage{array}

\makeatletter
\newcommand{\leqnomode}{\tagsleft@true}
\newcommand{\reqnomode}{\tagsleft@false}
\makeatother

\usepackage[UKenglish]{babel}
\usepackage{enumerate}
\usepackage{hyperref}
\usepackage{latexsym}
\usepackage{lmodern}
\usepackage{algorithm}
\usepackage{algpseudocode}
\usepackage{enumitem}

\theoremstyle{plain}
\newtheorem{theorem}{Theorem}[section]

\newtheorem{lemma}[theorem]{Lemma}
\newtheorem{corollary}[theorem]{Corollary}
\newtheorem{proposition}[theorem]{Proposition}

\theoremstyle{definition}
\newtheorem{defn}{Definition}

\theoremstyle{remark}
\newtheorem{remark}[theorem]{Remark}

\usepackage{tikz}
\usetikzlibrary{arrows,snakes,backgrounds}
\usetikzlibrary{decorations.pathmorphing}
\usepackage{graphicx}
\usepackage{caption}
\usepackage{subcaption}
\usepackage{subfloat}
\usepackage{wrapfig,lipsum}

\newcommand{\sgn}{\operatorname{sgn}}

\newtheorem{innercustomgeneric}{\customgenericname}
\providecommand{\customgenericname}{}
\newcommand{\newcustomtheorem}[2]{%
  \newenvironment{#1}[1]
  {%
   \renewcommand\customgenericname{#2}%
   \renewcommand\theinnercustomgeneric{##1}%
   \innercustomgeneric
  }
  {\endinnercustomgeneric}
}

\theoremstyle{plain}

\newlength{\problemoffset}
\newcommand{\decision}[3]{
\begin{list}{}{
\setlength{\leftmargin}{0.6cm}
\setlength{\rightmargin}{1cm}
\setlength{\parsep}{0pt}
\setlength{\itemsep}{2pt}
\setlength{\topsep}{4pt}
\setlength{\partopsep}{\itemsep}
}
\item
{\textsc{#1}}
\item
{{\textbf{INSTANCE:}} #2}
\item
{{\textbf{QUESTION:}} #3}
\end{list}
}

\newcustomtheorem{customlemma}{Lemma}

\newcommand{\R}{\mathbb{R}}

\usepackage[margin=2.25cm]{geometry}
\usepackage{marvosym}
\renewcommand{\tilde}[1]{\widetilde{#1}}

\def\longbox#1{\parbox{0.85\textwidth}{#1}}

\newcommand*\samethanks[1][\value{footnote}]{\footnotemark[#1]}

\newcolumntype{M}[1]{>{\centering\arraybackslash}p{#1}}
\newcolumntype{V}[1]{>{\centering\arraybackslash}m{#1}}

\DeclareMathOperator*{\argmax}{arg\,max}

\title{Sums of Separable and Quadratic Polynomials}

\author{Amir Ali Ahmadi\thanks{A. A. Ahmadi and C. Dibek are with the department of Operations Research and Financial Engineering,
		Princeton University, USA.
		Emails: {\tt\small aaa@princeton.edu}; {\tt\small cdibek@princeton.edu}} \ \ \ \ \ \
Cemil Dibek\samethanks \ \ \ \ \ \
Georgina Hall\thanks{G. Hall is with the department of Decision Sciences,
		INSEAD, France. Email: {\tt \small georgina.hall@insead.edu} \newline This work was partially supported by an AFOSR MURI award, the DARPA
Young Faculty Award, the Princeton SEAS Innovation Award, the NSF CAREER Award, the Google Faculty Award,
and the Sloan Fellowship.
}
}

\date{}
\begin{document}\maketitle

\begin{abstract}
We study \emph{separable plus quadratic} (SPQ) polynomials, i.e., polynomials that are the sum of univariate polynomials in different variables and a quadratic polynomial. Motivated by the fact that nonnegative separable and nonnegative quadratic polynomials are sums of squares, we study whether nonnegative SPQ polynomials are (i) the sum of a nonnegative separable and a nonnegative quadratic polynomial, and (ii) a sum of squares. We establish that the answer to question~(i) is positive for univariate plus quadratic polynomials and for convex SPQ polynomials, but negative already for bivariate quartic SPQ polynomials. We use our decomposition result for convex SPQ polynomials to show that convex SPQ polynomial optimization problems can be solved by ``small" semidefinite programs. For question~(ii), we provide a complete characterization of the answer based on the degree and the number of variables of the SPQ polynomial. We also prove that testing nonnegativity of SPQ polynomials is NP-hard when the degree is at least four. We end by presenting applications of SPQ polynomials to upper bounding sparsity of solutions to linear programs, polynomial regression problems in statistics, and a generalization of Newton's method which incorporates separable higher-order derivative information.\\

\end{abstract}

\section{Introduction}
A polynomial $p: \R^n \rightarrow \R$ with real coefficients is said to be \emph{nonnegative} if $p(x) \geq 0$ for all $x \in \R^n$ and a \emph{sum of squares} (sos) if there exist polynomials $q_1(x),\ldots,q_m(x)$ such that $p(x) = \sum_{i=1}^m q_i^2(x)$. It is clear that the set $\Sigma_{n,d}$ of sos polynomials of degree $d$ in $n$ variables is contained in the set $N_{n,d}$ of nonnegative polynomials of degree $d$ in $n$ variables. The question of equivalence between $N_{n,d}$ and $\Sigma_{n,d}$ is a classical problem of algebraic geometry which was resolved by Hilbert in 1888:

\begin{theorem}[\cite{Hilbert}]
$\Sigma_{n,d} = N_{n,d}$ if and only if $n=1$, or $d=2$, or $(n, d) = (2, 4)$.
\label{thm:Hilbert}
\end{theorem}
\reqnomode

As both nonnegative univariate ($n=1$) polynomials and nonnegative quadratic ($d=2$) polynomials are sums of squares, it is natural to wonder what would happen to the sum of a univariate and a quadratic polynomial. Would it also be the case that any nonnegative polynomial with such a structure admits a sum of squares representation? We aim to answer this and related questions in a more general setting where the univariate polynomial is replaced by a separable polynomial. This structure is captured in the following definition.
\begin{defn}
\label{defn:spq}
A polynomial $p: \R^n \rightarrow \R$ is \emph{separable plus quadratic (SPQ)} if $p(x) = s(x) + q(x)$, where $s(x)$ is a separable polynomial, i.e., $s(x)=\sum_{i=1}^n s_i(x_i)$ for some univariate polynomials $s_i:\R \rightarrow \R$, and $q(x)$ is a quadratic polynomial.
\end{defn}

Hilbert's proof of Theorem \ref{thm:Hilbert} exhibited no explicit examples of polynomials which are nonnegative but not sos. In fact, it took a further 80 years for such examples to emerge. Of particular interest are examples corresponding to the cases where $(n,d)=(2,6)$ and $(n,d)=(3,4)$ as they constitute the minimal cases for which $\Sigma_{n,d} \neq N_{n,d}$. These examples were produced by Motzkin~\cite{Motzkin} and Robinson~\cite{Robinson} respectively and are given below:
\begin{align}
&M(x_1, x_2) = x_1^4x_2^2 + x_1^2x_2^4 - 3x_1^2x_2^2 + 1, \label{ex:Motzkin}\\
&R(x_1, x_2, x_3) =  x_1^2(x_1 -1)^2 + x_2^2(x_2 -1)^2 + x_3^2(x_3 -1)^2 + 2x_1x_2x_3(x_1+x_2+x_3-2). \label{ex:Robinson}
\end{align}
Many other examples have appeared in the literature over the years; see, e.g., \cite{Reznick, ChoiLam1, ChoiLam2, ChoiLamRez80}. An interesting feature of existing examples such as (\ref{ex:Motzkin}) and (\ref{ex:Robinson}) is the presence of cross-terms of high degree. It is thus not immediately clear that examples of nonnegative but not sos polynomials are possible among SPQ polynomials, as their cross-terms have degree equal to $2$.

While questions about the relationship between sos and nonnegative polynomials had previously been the preserve of the mathematics community, the beginning of the 21st century saw a renewed interest in these questions originating from the optimization community. This was mainly due to two factors: first, the observation that many important problems in semialgebraic optimization can be reformulated as optimization problems over nonnegative polynomials; second, the discovery of a fundamental link between semidefinite programming and sos polynomials, as given below.
\begin{theorem}[\cite{ChoiLamRez, ParriloThesis}]
A polynomial $p(x)$ in $n$ variables and of degree $2d$ is sos if and only if there exists a (symmetric) positive semidefinite matrix $Q$ such that $p(x) = z(x)^TQz(x)$ where $z(x)$ is the vector of monomials of degree up to $d$, i.e., $z(x) = (1, x_1, \dots, x_n, x_1x_2, \dots, x_n^d)^T$.
\label{thm:sos_psd}
\end{theorem}
This theorem immediately leads to a semidefinite programming-based method for checking whether a polynomial is a sum of squares, and in fact, more interestingly, for optimizing a linear function over the intersection of the set of sos polynomials with an affine subspace. In sharp contrast, optimization over the set of nonnegative polynomials is intractable. Indeed, simply checking whether a polynomial of degree $4$ is nonnegative is NP-hard \cite{MurtKaba}.

The fact that optimization over the set of sos polynomials can be done using semidefinite programming has enabled wide-ranging applications. As alluded to before, numerous semialgebraic problems in applied and computational mathematics can be cast as optimization problems over the set of nonnegative polynomials; see, e.g., \cite{LasserreBook, BPT, ghsurvey}. While these problems are generally intractable to solve exactly, it is nevertheless possible to use sos polynomials as surrogates for nonnegative polynomials and, in view of Theorem~\ref{thm:sos_psd}, solve an approximation of the problem using semidefinite programming. It is thus increasingly relevant to study the relationship between nonnegative and sos polynomials under additional structure. This is what this paper proposes to do for polynomials with an SPQ structure. As mentioned previously, this is a very natural structure to consider in light of the first two equality cases in Theorem~\ref{thm:Hilbert}; it is also a structure of interest in various applications, as we see later. We further extend our study in this paper to understanding nonnegativity of \emph{convex} SPQ polynomials. This has implications for polynomial optimization problems involving such polynomials.

\subsection{Organization and main contributions}\label{subsec:organization}

The organization of the remainder of this paper is as follows. In Section \ref{sec:special_cases}, we study when nonnegative SPQ polynomials can be written as the sum of nonnegative univariate polynomials and a nonnegative quadratic polynomial. In Section \ref{subsec:separable}, we show that this is the case for nonnegative separable polynomials and nonnegative polynomials that are the sum of a univariate and a quadratic polynomial. In Section \ref{subsec:nonneg_sep_plus_nonneg_quad}, we show that this is not the case in general, and provide a minimal example where this decomposition fails to exist.

In Section \ref{sec:spq_polynomials}, we prove the analogue of Theorem \ref{thm:Hilbert} for SPQ polynomials. This involves constructing minimal examples of SPQ polynomials that are nonnegative but not sos (Section~\ref{subsec:minimal_cases}) and then generalizing these examples to higher degrees (Section~\ref{subsec:higher_n_d}). While the results of Section~\ref{sec:spq_polynomials} imply that testing nonnegativity of SPQ polynomials cannot always be accomplished via a sum of squares decomposition, they do not exclude the possibility of a polynomial-time algorithm for the task. In Section \ref{sec:complexity}, we give a proof of NP-hardness of deciding nonnegativity of degree-4 SPQ polynomials. This precludes a polynomial-time algorithm from existing, unless P$=$NP.

Section \ref{sec:convexity} focuses on convex SPQ polynomials. In Section \ref{subsec:check_convexity}, we show that the Hessian of a convex SPQ polynomial can be written as the sum of positive semidefinite univariate polynomial matrices. In Section \ref{subsec:cone_nonneg_sum}, we prove that any nonnegative convex SPQ polynomial can be written as the sum of a nonnegative separable and a nonnegative quadratic polynomial. In Section~\ref{subsec:spq_optimization}, we build on this result to show that polynomial optimization problems whose objective and constraint functions are given by convex SPQ polynomials can be solved via a single semidefinite program whose size is much smaller than that obtained via the first level of the Lasserre hierarchy. A procedure for extracting an optimal solution is also presented.

We conclude our paper in Section \ref{sec:applications} with three potential applications involving SPQ polynomials. In Section~\ref{subsec:lower_bounds_sparsity}, we use separable polynomials as a surrogate for the $\ell_0$-pseudonorm to obtain upper bounds on the sparsity of solutions of linear programs. These bounds improve on those given by the $\ell_1$-norm. As opposed to the $\ell_1$-based approach, our approach takes into consideration the problem data, and can produce problem-specific surrogates for the $\ell_0$-pseudonorm. In Section \ref{subsec:spq_regression}, we consider shape-constrained polynomial regression where we fit a convex SPQ polynomial to noisy evaluations of a convex function with low-degree interactions between variables. In Section \ref{subsec:newton_spq}, we propose a variant of Newton's method for minimizing a multivariate function that relies on local approximations by SPQ polynomials, instead of local quadratic approximations. We perform numerical experiments for all three applications highlighting the potential benefits of these approaches.

\section{Nonnegativity of Special Cases of SPQ Polynomials}\label{sec:special_cases}

We begin this section by considering two subsets of the set of SPQ polynomials: separable polynomials and univariate plus quadratic polynomials. We show that these polynomials are nonnegative if and only if they are sos. Our proof technique is similar for both results: it involves recasting nonnegative polynomials with these structures as sums of nonnegative univariate polynomials and a nonnegative quadratic polynomial. This approach motivates the question as to whether such a decomposition is always possible for nonnegative SPQ polynomials. We give a negative answer to this question in Section~\ref{subsec:nonneg_sep_plus_nonneg_quad}. Interestingly, such a decomposition can fail even when the existence of a sum of squares decomposition is guaranteed.

\subsection{Separable polynomials and univariate plus quadratic polynomials}\label{subsec:separable}

We start by examining nonnegativity of separable polynomials.

\begin{lemma}
Every nonnegative separable polynomial can be written as the sum of nonnegative univariate polynomials. In particular, a separable polynomial is nonnegative if and only if it is sos.
\label{lem:separable_nonneg}
\end{lemma}

\begin{proof}
Let $p(x) = \sum_{i=1}^n p_i(x_i)$ be a nonnegative separable polynomial. Let $x_i^* \in \R$ be a global minimum\footnote{As $p(x)$ is nonnegative, each function $p_i(x_i)$ is lower bounded, and being a univariate polynomial, its infimum is attained.} of $p_i(x_i)$ and let $x^* \in \R^n$ be the vector $x^* = (x_1^*, x_2^*, \dots, x_n^*)^T$. We have
\begin{equation}
\hspace{5cm}
p(x) = \sum\limits_{i=1}^{n} \Big (p_i(x_i) - p_i(x_i^*) \Big ) + c,
\label{eq:sep_nonnega}
\end{equation}
where the constant $c$ is nonnegative since $c = p(x^*)$ and as $p(x^*) \geq 0$ by the assumption of nonnegativity of $p(x)$. Equation~\eqref{eq:sep_nonnega} thus proves the first statement. The second statement follows as an immediate corollary of Theorem~\ref{thm:Hilbert}.
\end{proof}

We now prove a similar result for univariate plus quadratic polynomials, although the proof is slightly more involved.

\leqnomode
\begin{theorem}
Let $p(x)$ be a polynomial that can be written as the sum of a univariate and a quadratic polynomial. If $p(x)$ is nonnegative, then it can be written as the sum of a nonnegative univariate and a nonnegative quadratic polynomial. In particular, $p(x)$ is nonnegative if and only if it is sos.
\label{thm:nonneg_uni_plus_quadratic}
\end{theorem}

\begin{proof}
Let $p(x)$ be a nonnegative polynomial in $n$ variables and of degree $d$ that can be written as the sum of a univariate and a quadratic polynomial. We may assume that the variable whose degree in $p(x)$ is higher than 2 is $x_1$. (If there is no such variable, the claim is trivial.) By pushing, if necessary, the constant, $x_1$, and $x_1^2$ terms from the quadratic part into the univariate part, we may also assume that $p(x)$ can be written as $p(x) = \bar{x}^T A \bar{x} + u(x_1)$, where $\bar{x} = (1, x_1, x_2, \dots x_n)^T$, $u(x_1)$ is a univariate polynomial of degree $d$, and the matrix $A$ is of the following structure:
$$
A = 
  \begin{pmatrix}
      0 & 0 & a_2 & a_3 & \dots & a_n \\
      0 & 0 & b_2 & b_3 & \dots & b_n \\
      a_2 & b_2 & c_2 & e_{23} & \dots & e_{2n} \\
      a_3 & b_3 & e_{23} & c_3 &  &  \vdots \\
      \vdots & \vdots & \vdots &  & \ddots &  \\
      a_n & b_n & e_{2n} & \dots &  & c_n \\
  \end{pmatrix}.
$$
Let $C$ denote the $(n-1) \times (n-1)$ matrix obtained from $A$ by deleting the first and second rows and columns. We observe that $p(x)$ can also be written as $p(x) = \tilde{x}^T A(x_1) \tilde{x}$, where $\tilde{x} = (1, x_2, \dots, x_n)^T$ and
$$
A(x_1) = 
  \begin{pmatrix}
      u(x_1) & a_2 + b_2x_1 & a_3 + b_3x_1 & \dots & a_n + b_nx_1 \\
      a_2 + b_2x_1 & c_2 & e_{23} & \dots & e_{2n} \\
      a_3 + b_3x_1 & e_{23} & c_3 &  &  \vdots \\
      \vdots & \vdots & \vdots & \ddots &  &  \\
      a_n +b_nx_1 & e_{2n} & \dots &  & c_n \\
  \end{pmatrix}.
$$
Since $p(x)$ is nonnegative, the matrix $A(x_1)$ is positive semidefinite for all $x_1 \in \mathbb{R}$, and therefore $C$ is positive semidefinite. Let $C^{\dagger}$ be the pseudo-inverse of $C$, $a = (a_2, \dots, a_n)^T$, $b = (b_2, \dots, b_n)^T$, and $c = (c_2, \dots, c_n)^T$. We define $t(x_1)$ to be the univariate quadratic polynomial given by 
$$t(x_1) = (a+bx_1)^T C^{\dagger} (a+bx_1).$$
Let $I$ denote the $(n-1) \times (n-1)$ identity matrix. We now prove two claims concerning $t(x_1)$.

\vspace{-0.4cm}

\begin{equation}
\longbox{{\it  The polynomial $u(x_1) - t(x_1)$ is nonnegative, $(I - CC^{\dagger}) a = 0$, and $(I - CC^{\dagger}) b = 0$.}}\tag{$2.2.1$}
\label{eq:C_inverse}
\end{equation}
We recall the generalized Schur complement (see, e.g., \cite{Boyd}): A symmetric matrix
$M = \begin{pmatrix}
     X & Y \\
     Y^T & Z
  \end{pmatrix}$
is positive semidefinite if and only if $Z$ is positive semidefinite, $(I - ZZ^{\dagger}) Y^T = 0$, and the matrix $X - YZ^{\dagger}Y^T$ is positive semidefinite. Since the matrix $A(x_1)$ is positive semidefinite for every $x_1 \in \mathbb{R}$, it follows that $u(x_1) - (a+bx_1)^T C^{\dagger} (a+bx_1) = u(x_1) - t(x_1)$ is nonnegative and that $(I - CC^{\dagger}) (a+bx_1) = 0$ for every $x_1 \in \mathbb{R}$. Hence, $(I - CC^{\dagger}) a = 0$ and $(I - CC^{\dagger}) b = 0$.

\vspace{-0.4cm}

\begin{equation}
\longbox{{\it The quadratic polynomial $\bar{x}^T A \bar{x} + t(x_1)$ is nonnegative.}}
\label{eq:quad_part_nonneg}\tag{$2.2.2$}
\end{equation}
Observe that 
$t(x_1) = (a+bx_1)^T C^{\dagger} (a+bx_1) = (1 \:\:\: x_1)^T 
\begin{pmatrix}
      a^TC^{\dagger}a & a^TC^{\dagger}b \\
      b^TC^{\dagger}a & b^TC^{\dagger}b \\
 \end{pmatrix} (1 \:\:\: x_1)$.

\noindent We then have $\bar{x}^T A \bar{x} + t(x_1) = \bar{x}^T L \bar{x}$, where
$$
L = 
\begin{pmatrix}
a^TC^{\dagger}a & a^TC^{\dagger}b & a^T \\
b^TC^{\dagger}a & b^TC^{\dagger}b & b^T \\
a & b & C
\end{pmatrix}.
$$
Note that $C$ is positive semidefinite, $(I - CC^{\dagger}) (a \quad b) = 0$, and that the matrix 
$$\begin{pmatrix}
      a^TC^{\dagger}a & a^TC^{\dagger}b \\
      b^TC^{\dagger}a & b^TC^{\dagger}b \\
 \end{pmatrix} - (a \:\:\: b)^T C^{\dagger} (a \:\:\: b)$$ is positive semidefinite
as it equals the zero matrix. By the generalized Schur complement, it follows that $L$ is positive semidefinite. Therefore, the quadratic polynomial $\bar{x}^T A \bar{x} + t(x_1)$ is nonnegative.

\medskip

To finish the proof of the theorem, we now write $p(x)$ as follows:
$$p(x) = \bar{x}^T A \bar{x} + u(x_1) = \bar{x}^T A \bar{x} + t(x_1) + u(x_1) - t(x_1).$$
By \eqref{eq:C_inverse}, the univariate polynomial $u(x_1) - t(x_1)$ is nonnegative, and by \eqref{eq:quad_part_nonneg}, the quadratic polynomial $\bar{x}^T A \bar{x} + t(x_1)$ is nonnegative.
\end{proof}

\begin{remark}
The second (and weaker) assertion in the statement of Theorem \ref{thm:nonneg_uni_plus_quadratic}, i.e., that $p(x)$ is nonnegative if and only if it is sos, can also be obtained via the following theorem (see \cite{ChoiLamRez80} for a self-contained proof and \cite{AIP} for a discussion of related literature). Recall that a \emph{form} (or a homogeneous polynomial) is a polynomial where all the monomials have the same degree.
\vspace{-6.5mm}
\begin{itemize}
\item[]
\begin{theorem}[\cite{AIP, ChoiLamRez80}]
Let $f(u_1, u_2, v_1, \dots, v_m)$ be a form in the variables $u = (u_1, u_2)$ and $v = (v_1, \dots, v_m)$ that is a quadratic form in $v$ for fixed $u$ and a form (of any degree) in $u$ for fixed $v$. Then, $f(u_1, u_2, v_1, \dots, v_m)$ is nonnegative if and only if it is sos.
\label{thm:biform}
\end{theorem}
\end{itemize}

\noindent Following the notation in the proof of Theorem \ref{thm:nonneg_uni_plus_quadratic}, we have $p(x) = \tilde{x}^T A(x_1) \tilde{x}$. 
Since $p(x)$ is nonnegative, the matrix $A(x_1)$ is positive semidefinite for every $x_1 \in \mathbb{R}$. Let $\tilde{A}(x_1, x_{n+1})$ be the $n \times n$ matrix whose entries are obtained by homogenizing (see, e.g., \cite{Reznick}) the entries of $A(x_1)$. It is easy to see that $\tilde{x}^T \tilde{A}(x_1, x_{n+1}) \tilde{x}$ is then the form obtained by homogenizing $p(x) = \tilde{x}^T A(x_1) \tilde{x}$ and is therefore nonnegative. Now we can employ Theorem \ref{thm:biform} with $(u_1, u_2) = (x_1, x_{n+1})$ and $(v_1, \dots, v_n) = (1, x_2, \dots, x_n)$ to deduce that $\tilde{x}^T \tilde{A}(x_1, x_{n+1}) \tilde{x}$ is sos. Upon dehomogenizing by setting $x_{n+1} = 1$, we conclude that $p(x) = \tilde{x}^T A(x_1) \tilde{x}$ is sos. This completes the proof. However, this proof does not show that $p(x)$ can be written as the sum of a nonnegative univariate and a nonnegative quadratic polynomial.
\end{remark}

\begin{remark}
The proof of Theorem \ref{thm:nonneg_uni_plus_quadratic} suggests a simple algorithm for checking nonnegativity of a univariate plus quadratic polynomial $p(x)$ that does not require semidefinite programming. Using the notation of the proof, we can check nonnegativity of $p(x)$ by equivalently testing that $C$ is positive semidefinite, $(I - CC^{\dagger}) a = 0$, $(I - CC^{\dagger}) b = 0$, and that the univariate polynomial $u(x_1) - t(x_1)$ is nonnegative\footnote{Testing nonnegativity of a (nonconstant) univariate polynomial can be done, e.g., by checking that all real roots have even multiplicity and that the polynomial takes a positive value at an arbitrary point (which is not a root).}.
\end{remark}

\subsection{Sums of nonnegative separable and nonnegative quadratic polynomials}\label{subsec:nonneg_sep_plus_nonneg_quad}

The prevailing idea in Lemma \ref{lem:separable_nonneg} and Theorem \ref{thm:nonneg_uni_plus_quadratic} is to write a nonnegative polynomial as the sum of nonnegative univariate and nonnegative quadratic polynomials. It is therefore natural to more generally investigate the relationship between the following three sets of polynomials:
\begin{itemize}
\item $N^{SPQ}_{n,d}$: the set of nonnegative SPQ polynomials in $n$ variables and degree $d$,

\item $\Sigma^{SPQ}_{n,d}$: the set of sos SPQ polynomials in $n$ variables and degree $d$,

\item $(N^S+N^Q)_{n,d}$: the set of SPQ polynomials in $n$ variables and degree $d$ that can be written as the sum of a nonnegative separable polynomial and a nonnegative quadratic polynomial.\footnote{By Lemma \ref{lem:separable_nonneg}, $(N^S+N^Q)_{n,d}$ is the same set as the set of polynomials in $n$ variables and degree $d$ that can be written as the sum of nonnegative \emph{univariate} polynomials and a nonnegative quadratic polynomial. Interestingly, $(N^S+N^Q)_{n,d}$ is different from the set of SPQ polynomials in $n$ variables and degree $d$ that can be written as the sum of a nonnegative separable polynomial and a nonnegative quadratic \emph{form} (see Remark \ref{rem:minimal_2_2} and Lemma \ref{lem:minimal_2_2}).}
\end{itemize}
We have
$$(N^S+N^Q)_{n,d} \: \subseteq \: \Sigma^{SPQ}_{n,d} \: \subseteq \: N^{SPQ}_{n,d},$$
where the second inclusion is evident and the first follows from Lemma \ref{lem:separable_nonneg} and Theorem \ref{thm:Hilbert}. One might be tempted to show that $(N^S+N^Q)_{n,d} = N^{SPQ}_{n,d}$, which would imply $\Sigma^{SPQ}_{n,d} = N^{SPQ}_{n,d}$ and prove that a nonnegative SPQ polynomial is sos. However, this approach would not work, as the following lemma shows that even equality between $(N^S+N^Q)_{n,d}$ and $\Sigma^{SPQ}_{n,d}$ does not hold. We present an example of an \emph{sos} SPQ polynomial in $2$ variables and degree $4$ that cannot be written as the sum of a nonnegative separable and a nonnegative quadratic polynomial. Since we have $(N^S+N^Q)_{n, d} = \Sigma^{SPQ}_{n, d}$ for $n=1$ and for $d=2$, this example is minimal in both degree and dimension.

\begin{lemma}
The bivariate quartic polynomial $$p(x) = x_1^4 - x_1^2 + 2x_1 + x_2^4 - x_2^2 - 2x_2 + \frac{12}{5} - 2x_1x_2$$ belongs to $\Sigma^{SPQ}_{2, 4} \setminus (N^S+N^Q)_{2, 4}$.
\label{lem:counterexample_for_nonneg_sum}
\end{lemma}

\begin{proof}
Clearly, the polynomial $p(x)$ is SPQ. We prove the two claims separately.
\leqnomode
\begin{equation}
\longbox{{\it The polynomial $p(x)$ is sos.}}
\label{eq:ex_p_sos}\tag{$2.6.1$}
\end{equation}
One can observe that $p(x) = z(x)^T Q z(x)$ where
$z(x) =
(1, x_1, x_2, x_1^2, x_1x_2, x_2^2)^T
$
and $Q$ is the matrix
$$
\hspace{-1.2cm}
\small{Q \: = \: \frac{1}{240}
\left(\begin{array}{cccccc} 576 & 240 & -240 & -205 & -108 & -205\\ 240 & 170 & -132 & 0 & -99 & -99\\ -240 & -132 & 170 & 99 & 99 & 0\\ -205 & 0 & 99 & 240 & 0 & -42\\ -108 & -99 & 99 & 0 & 84 & 0\\ -205 & -99 & 0 & -42 & 0 & 240 \end{array}\right)}.
$$
It can be easily checked that $Q$ is positive semidefinite (in fact, positive definite\footnote{Whenever we state a matrix is positive definite, this claim is supported by a rational $LDL^T$ factorization of the matrix. The operations showing that $p(x) = z(x)^T Q z(x)$ and that $Q$ is positive definite can be found online at \\ \url{http://colab.research.google.com/github/cdibek/spq_polynomials/blob/main/Lemma_2_6_Proof.ipynb}.}). Therefore, $p(x)$ is sos (cf. Theorem \ref{thm:sos_psd}).

\vspace{-0.3cm}

\newpage
\begin{equation}
\longbox{{\it The polynomial $p(x)$ cannot be written as the sum of a nonnegative separable and a nonnegative quadratic polynomial.}}
\label{eq:ex_p_not_nonneg_sum}\tag{$2.6.2$}
\end{equation}
Assume for the sake of contradiction that $p(x) = s(x) + q(x)$ where $s(x)$ is a nonnegative separable polynomial and $q(x)$ is a nonnegative quadratic polynomial. By Lemma \ref{lem:separable_nonneg}, $s(x)$ can be written as $s(x) = s_1(x_1) + s_2(x_2)$ where $s_1(x_1), s_2(x_2)$ are nonnegative univariate polynomials. We observe that for some choice of parameters $a, b, c, d, e_1, e_2 \in \mathbb{R}$, the polynomial $p(x)$ is the sum of the following three polynomials:
\begin{equation*}
\begin{aligned}
s_1(x_1) & = x_1^4 - (a+1)x_1^2 - (2c-2)x_1 + e_1, \\
s_2(x_2) & = x_2^4 - (b+1)x_2^2 - (2d+2)x_2 + e_2, \\
q(x) & = ax_1^2 - 2x_1x_2 + bx_2^2 + 2cx_1 + 2dx_2 + \frac{12}{5} - e_1 - e_2.
\end{aligned}
\end{equation*}
The nonnegative polynomials $s_1(x_1), s_2(x_2), q(x)$ are sos since they are univariate or quadratic. Thus, by Theorem \ref{thm:sos_psd}, there exist positive semidefinite matrices $A, B, C$ such that 
$$s_1(x_1) = z_1(x_1)^TAz_1(x_1), \quad \quad s_2(x_2) = z_2(x_2)^TBz_2(x_2), \quad \quad q(x) = z(x)^TCz(x),$$
where $z_1(x_1) = (1, x_1, x_1^2)^T$, $z_2(x_2) = (1, x_2, x_2^2)^T$, $z(x) = (1, x_1, x_2)^T$, and
$$
\small{
A = 
  \begin{pmatrix}
     e_1 & -c+1 & A_{13} \\
    -c+1 & -2A_{13}-a-1 & 0 \\
     A_{13} & 0 & 1
  \end{pmatrix}, \:
B = 
  \begin{pmatrix}
     e_2 & -d-1 & B_{13} \\
     -d-1 & -2B_{13}-b-1 & 0 \\
     B_{13} & 0 & 1
  \end{pmatrix}, \:
C = 
  \begin{pmatrix}
     \frac{12}{5}-e_1-e_2 & c & d \\
     c & a & -1 \\
     d & -1 & b
  \end{pmatrix}}.
$$
It follows that the matrix
$$
\small{
E = A + C =
  \begin{pmatrix}
     \frac{12}{5} - e_2 & 1 & A_{13} + d\\
     1 & -2A_{13}-1 & -1 \\
     A_{13} + d & -1 & b+1
  \end{pmatrix}}
$$
is positive semidefinite. In the rest of the proof, we show that the matrices $B$ and $E$ cannot be positive semidefinite at the same time, leading to a contradiction. First, for notational convenience, we do the following change of variables: $b+1 \rightarrow v$, $d \rightarrow y$, $e_2 \rightarrow u$, $-A_{13} \rightarrow w$, and $-B_{13} \rightarrow t$. We have
$$
\small{
B = 
  \begin{pmatrix}
     u & -y-1 & -t \\
    -y-1 & 2t-v & 0 \\
     -t & 0 & 1
  \end{pmatrix}, \quad
E = 
  \begin{pmatrix}
     \frac{12}{5} - u & 1 & y-w\\
     1 & 2w-1 & -1 \\
     y-w & -1 & v
  \end{pmatrix}}.
$$
Now, consider the following two matrices:
$$
\small{
\tilde B = 
  \begin{pmatrix}
     72 & 56 & 60 \\
     56 & 60 & 46 \\
     60 & 46 & 51
  \end{pmatrix}, \quad \quad
\tilde E = 
  \begin{pmatrix}
     72 & -27 & 56 \\
     -27 & 56 & 5 \\
      56 & 5 & 60
  \end{pmatrix}}.
$$
It can easily be checked that the matrices $\tilde B$ and $\tilde E$ are positive definite. Since the matrices $B, E, \tilde B, \tilde E$ are all positive semidefinite, we must have $\text{Tr}(B \tilde B) + \text{Tr} (E \tilde E) \geq 0$, where for a matrix $M$, the notation $\text{Tr} (M)$ denotes the trace of $M$. We have
$$\text{Tr} (B \tilde B) + \text{Tr} (E \tilde E) = 72 u - 60v - 112y - 61 + 60v - 72u + 112y + 52.8 = -8.2,$$
a contradiction.
\end{proof}

Given the polynomial $p(x)$ given in Lemma \ref{lem:counterexample_for_nonneg_sum}, it is straightforward to construct polynomials in $\Sigma^{SPQ}_{n,d} \setminus (N^S+N^Q)_{n,d}$ for any $n \geq 2$ and $d \geq 4$. We omit the proof of this construction as it is very similar in style to the proof of Theorem~\ref{thm:higher_deg_dim} in the next section.




\begin{remark}
\label{rem:minimal_2_2}
In view of the definition of an SPQ polynomial (cf. Definition \ref{defn:spq}), observe that any SPQ polynomial can be written as the sum of a separable polynomial and a quadratic form (since we may always push the constant term and the degree-1 terms into the separable part). However, this distinction becomes more subtle when we consider the cone $(N^S+N^Q)_{n,d}$. Let $(N^S+N^{Q_f})_{n,d}$ denote the set of polynomials in $n$ variables and degree $d$ that can be written as the sum of a nonnegative separable polynomial and a nonnegative quadratic form. Although the definition of an SPQ polynomial remains the same when we replace ``quadratic polynomial" with ``quadratic form", the cones $(N^S+N^Q)_{n,d}$ and $(N^S+N^{Q_f})_{n,d}$ are not the same. While the inclusion $(N^S+N^{Q_f})_{n,d} \subseteq (N^S+N^Q)_{n,d}$ is clear, the following example shows that the converse inclusion does not hold. This example is minimal in degree and dimension as it belongs to $(N^S+N^Q)_{2,2} \setminus (N^S+N^{Q_f})_{2,2}$ and since $(N^S+N^{Q_f})_{n,d} = (N^S+N^Q)_{n,d}$ for $n=1$.
\end{remark}

\begin{lemma}
\label{lem:minimal_2_2}
The bivariate quadratic polynomial
$$p(x) = 8x_1^2 - 4x_1 + 2x_2^2 - 2x_2 + 3 + 8x_1x_2$$
belongs to $(N^S+N^Q)_{2,2} \setminus (N^S+N^{Q_f})_{2,2}$.
\end{lemma}

\begin{proof}
The polynomial $p(x)$ belongs to $(N^S+N^Q)_{2,2}$ because it is a nonnegative quadratic polynomial. 
To see that $p(x)$ is nonnegative, observe that $$p(x) = \Big (\frac{\sqrt{5} +1}{2} -2x_1 - x_2 \Big)^2 + \Big (\frac{\sqrt{5} - 1}{2} + 2x_1 + x_2 \Big)^2.$$

Assume for the sake of contradiction that $p(x) = s(x) + q(x)$, where $s(x)$ is a nonnegative separable polynomial and $q(x)$ is a nonnegative quadratic form. By Lemma \ref{lem:separable_nonneg}, the polynomial $s(x)$ can be written as $s(x) = s_1(x_1) + s_2(x_2)$ where $s_1(x_1), s_2(x_2)$ are nonnegative univariate polynomials. We observe that for some choice of parameters $a, b, c, d, e, f \in \mathbb{R}$, the polynomial $p(x)$ is the sum of the following three polynomials:
$$
s_1(x_1) = ax_1^2 - 4x_1 + b, \quad \quad
s_2(x_2) = cx_2^2 - 2x_2 + d, \quad \quad
q(x) = ex_1^2 + 8x_1x_2 + fx_2^2.
$$
The nonnegative polynomials $s_1(x_1), s_2(x_2), q(x)$ are sos since they are quadratic. By Theorem \ref{thm:sos_psd}, there exist positive semidefinite matrices $A, B, C$ such that
$$s_1(x_1) = z_1(x_1)^T A z_1(x_1), \quad \quad s_2(x_2) = z_2(x_2)^T B z_2(x_2), \quad \quad q(x) = z(x)^T C z(x),$$
where $z_1(x_1) = (1, x_1)^T$, $z_2(x_2) = (1, x_2)^T$, $z(x) = (x_1, x_2)^T$, and
$$
A = 
  \begin{pmatrix}
     b & -2 \\
    -2 & a
  \end{pmatrix}, \quad \quad
B = 
  \begin{pmatrix}
     d & -1 \\
     -1 & c
  \end{pmatrix}, \quad \quad
C = 
  \begin{pmatrix}
     e & 4 \\
     4 & f
  \end{pmatrix}.
$$
We have
\begin{itemize}
\itemsep0em
\item $b + d = 3$, $a + e = 8$, $c + f = 2$,
\item $a, b, c, d, e, f \geq 0$, $ab \geq 4$, $cd \geq 1$, and $ef \geq 16$,
\end{itemize}
where the equations in the first item hold since $p(x) = s_1(x_1) + s_2(x_2) + q(x)$ and the inequalities in the second item hold since $A, B, C$ are positive semidefinite. Note that $e \leq 8$ and $f \leq 2$ since $a, c \geq 0$. As we also know that $e, f \geq 0$ and $ef \geq 16$, we conclude that $e = 8$ and $f = 2$. Therefore, $a = 0$, which contradicts the inequality $ab \geq 4$.
\end{proof}

\section{Nonnegative SPQ Polynomials That Are Not Sums of Squares}\label{sec:spq_polynomials}

We have established that already when $(n,d)=(2,4)$, not every nonnegative SPQ polynomial can be written as the sum of a nonnegative separable and a nonnegative quadratic polynomial (cf. see Lemma \ref{lem:counterexample_for_nonneg_sum}). However, this does not rule out the possibility of a sum of squares decomposition for nonnegative SPQ polynomials. The main result of this section is the following theorem, which precisely characterizes degrees and dimensions where nonnegative SPQ polynomials are sos.

\begin{theorem}\label{thm:main_thm}
$\Sigma^{SPQ}_{n,d} = N^{SPQ}_{n,d}$ if and only if $n=1$, or $d=2$, or $(n, d) = (2, 4)$.
\end{theorem}

The equality $\Sigma^{SPQ}_{n,d} = N^{SPQ}_{n,d}$ holds in these cases by virtue of Theorem \ref{thm:Hilbert} (this is true for any polynomial with $n=1$, or $d=2$, or $(n,d)=(2,4)$, so in particular for SPQ polynomials). It remains to prove that the equality does not hold in the other cases. To show this, we present explicit examples of nonnegative SPQ polynomials that are not sos for the minimal cases, that is $(n,d)=(2,6)$ and $(n,d)=(3,4)$. These are given in Section \ref{subsec:minimal_cases}. We then show how to generalize these examples to higher degrees and dimensions in Section \ref{subsec:higher_n_d}.

\subsection{Minimal cases}\label{subsec:minimal_cases}

The next two theorems present examples of nonnegative SPQ polynomials that are not sos in the minimal cases.



\begin{theorem}
The bivariate sextic polynomial
$$p(x) = 17x_1^6 - 20x_1^4 + 7x_1^2 + 18x_1 + 18x_2^4 - 19x_2^2 - 19x_2 + 21 - 20x_1x_2$$
belongs to $N^{SPQ}_{2,6} \setminus \Sigma^{SPQ}_{2,6}$.
\label{thm:counter_example_2_6}
\end{theorem}

\begin{theorem}
The trivariate quartic polynomial
$$p(x) = x_1^4 + 2x_1^2 + x_2^4 + 2x_2^2 + x_3^4 + 2x_3^2 + \frac{9}{4} + 8x_1x_2 + 8x_1x_3 + 8x_2x_3$$
belongs to $N^{SPQ}_{3,4} \setminus \Sigma^{SPQ}_{3,4}$.
\label{thm:counter_example_3_4}
\end{theorem}

\begin{proof}[Proof of Theorem \ref{thm:counter_example_2_6}]
Clearly, the polynomial $p(x)$ is SPQ. The proof that $p(x)$ is not sos is done via a separating hyperplane argument. More precisely, we present a member $\mu$ of the cone dual to $\Sigma^{SPQ}_{2,6}$ whose inner product with the coefficients of $p(x)$ is negative. We fix the following monomial ordering:
\begin{align*}
v =
(1,  x_1,  x_2,  x_1^2,  x_1x_2,  x_2^2,  x_1^3,  x_1^2x_2,  x_1x_2^2, x_2^3, x_1^4, x_1^3x_2, x_1^2x_2^2, x_1x_2^3, x_2^4, x_1^5, x_1^4x_2, x_1^3x_2^2, x_1^6)^T.
\end{align*}
Let the vector of coefficients of $p(x)$ in the above monomial ordering be denoted by
\begin{align*}
\overrightarrow{p} = (21, 18, -19, 7, -20, -19, 0, 0, 0, 0, -20, 0, 0, 0, 18, 0, 0, 0, 17)^T.
\end{align*}
One can verify\footnote{All the computations in this proof are carried out over rational numbers and can be verified via the following link: \url{http://colab.research.google.com/github/cdibek/spq_polynomials/blob/main/Theorem_3_2_Proof.ipynb}.} that the vector 
\begin{align*}
\mu = (& 11156, -2031, 8817, 4897, -127, 8436, -1457, 3005, -292, 7015, 3302, -37, 3639, 759, \\ & 6730, -1105, 1873, -274, 2245)^T
\end{align*}
satisfies $\langle \mu, \overrightarrow{p} \rangle = -5 < 0$. We claim that for any sos polynomial $q(x)$ containing only the monomials in $v$, we should have $\langle \mu, \overrightarrow{q} \rangle \geq 0$, where $\overrightarrow{q}$ denotes the coefficients of $q(x)$ listed according to the ordering in $v$. Indeed, if $q(x)$ is sos, by Theorem \ref{thm:sos_psd}, it can be written as
$$q(x) = z(x)^TQz(x) = \text{Tr} \Big (Q \: z(x)z(x)^T \Big )$$
for some positive semidefinite matrix $Q$ and vector of monomials\footnote{If $\mathcal{N}$ is the Newton polytope of $q(x)$ (i.e., the convex hull of the exponent vectors of the 19 monomials in $v$), then the extreme points of $\mathcal{N}$ are $(0, 0), (6, 0), (0, 4)$. Hence, if $q(x) = \sum_i q_i^2(x)$ is a sum of squares, then the polynomials $q_i(x)$ are in the subspace spanned by the monomials with exponent vectors in $\frac{1}{2} \mathcal{N}$ (see, e.g., \cite{Reznick2}).}
$$z(x) = (1, x_1, x_2, x_1^2, x_1x_2, x_2^2, x_1^3)^T.$$
It is not difficult to see that
$$\langle \mu, \overrightarrow{q} \rangle =  \text{Tr}\left (Q \big (z(x)z(x)^T \big) |_\mu \right),$$
where 
$$
\big (z(x)z(x)^T \big)|_\mu = \scriptsize{
\left(\begin{array}{ccccccc}
11156 & -2031 & 8817 & 4897 & -127 & 8436 & -1457 \\ 
-2031 & 4897 & -127 & -1457 & 3005 & -292 & 3302 \\ 
8817 & -127 & 8436 & 3005 & -292 & 7015 & -37 \\ 
4897 & -1457 & 3005 & 3302 & -37 & 3639 & -1105 \\ 
-127 & 3005 & -292 & -37 & 3639 & 759 & 1873 \\ 
8436 & -292 & 7015 & 3639 & 759 & 6730 & -274 \\ 
-1457 & 3302 & -37 & -1105 & 1873 & -274 & 2245 
\end{array}\right)
}
$$
is the matrix where each monomial in $z(x)z(x)^T$ is replaced with the corresponding element of the vector $\mu$. We can check that the matrix $\big (z(x)z(x)^T \big)|_{\mu}$ is positive definite.\footnote{The proof of positive definiteness of $\big (z(x)z(x)^T \big)|_{\mu}$ is done by a rational $LDL^T$ factorization and can be found at \\ \url{http://colab.research.google.com/github/cdibek/spq_polynomials/blob/main/Theorem_3_2_Proof.ipynb}.} Now, since $Q$ and $\big (z(x)z(x)^T \big)|_\mu$ are positive semidefinite, we have $\langle \mu, \overrightarrow{q} \rangle =  \text{Tr} \left ( Q \big (z(x)z(x)^T \big)|_\mu \right ) \geq 0$. This completes the proof that $p(x)$ is not sos.

Next, we show that $p(x)$ is nonnegative. Recall that a function $f: \R^n \rightarrow \R$ is \emph{coercive} if for every sequence $\{x_k\}$ with $\{|| x_k ||\} \rightarrow +\infty$, we have $\{ f(x_k) \} \rightarrow +\infty$. It is easy to see that a continuous coercive function achieves its infimum on a closed set (see, e.g., Appendix A.2 of \cite{Bertsekas}).
Hence, a coercive polynomial always has a global minimum, and by the first order necessary condition for optimality, the gradient of the polynomial must vanish at all global minima. It follows that a coercive polynomial is nonnegative if and only if it is nonnegative at the points where its gradient vanishes. In the rest of the proof, we show that $p(x)$ is coercive, and that $p(x) \geq 0$ for all $x \in \R^2$ where the gradient $\nabla p(x) = 0$.

\vspace{-0.4cm}

\leqnomode
\begin{equation}
\longbox{{\it $p(x)$ is coercive.}}
\label{eq:p1_coercive}\tag{$3.2.1$}
\end{equation}
We write $p(x) = q_1(x_1) + q_2(x_2) + r(x)$, where
\begin{equation*}
\begin{aligned}
\hspace{4.5cm}
q_1(x_1) & = 17 x_1^6 - 20 x_1^4 + 6 x_1^2 + 18 x_1, \\
q_2(x_2) & = 17 x_2^4 - 19 x_2^2 - 19 x_2 + 21, \\
r(x) & = x_1^2 - 20 x_1x_2 + x_2^4.
\end{aligned}
\end{equation*}
Since $q_1(x_1), q_2(x_2)$ are univariate polynomials of even degree with a positive leading coefficient, they are coercive. Hence, the polynomial $q_1(x_1) + q_2(x_2)$ is coercive as well since it is the sum of two coercive univariate polynomials. Now, observe that $r(x)$ is bounded below:
$$r(x) = (x_1 - 10 x_2)^2 + (x_2^2 - 50)^2 -2500 \geq -2500, \space \forall x \in \R^2.$$
Therefore, $p(x) = q_1(x_1) + q_2(x_2) + r(x)$ is coercive since it is the sum of a coercive polynomial and a polynomial that is bounded below.

\vspace{-0.3cm}

\begin{equation}
\longbox{{\it $p(x) \geq 0$ for every $x \in \R^2$ that satisfies $\nabla p(x) = 0$.}}
\label{eq:p1_nonnegative}\tag{$3.2.2$}
\end{equation}
Observe that existence of two polynomials $\ell_1(x)$ and $\ell_2(x)$ which make the polynomial
$$p(x) - \Big (\ell_1(x), \: \ell_2(x) \Big)^T \: \nabla p(x)$$
sos would imply the claim (see \cite{NiDeSt} for a more in-depth treatment of this approach to proving nonnegativity). We show that $\ell_1(x) = 0$ and $\ell_2(x) = - \frac{1}{25} x_2^3$ satisfy this property. Indeed, with this choice, we have $p(x) - \big (\ell_1(x), \: \ell_2(x) \big)^T \: \nabla p(x) = \tilde z(x)^T \tilde Q \tilde z(x)$, where
$$\tilde z(x) = (1, x_1, x_2, x_1^2, x_1x_2, x_2^2, x_1^3, x_1^2x_2, x_1x_2^2, x_2^3)^T,$$
and $\tilde Q$ is the following matrix:
$$\hspace{-1cm} 
\tilde Q = \scriptsize{\frac{1}{300}
\left(\begin{array}{cccccccccc} 6300 & 2700 & -2850 & -1329 & -1635 & -4037 & 90 & 1282 & 246 & -828\\ 2700 & 4758 & -1365 & -90 & -2123 & -1533 & -3565 & 287 & -69 & -448\\ -2850 & -1365 & 2374 & 841 & 1287 & 714 & -238 & -993 & -151 & 277\\ -1329 & -90 & 841 & 1130 & -49 & 228 & 0 & -1050 & -147 & 489\\ -1635 & -2123 & 1287 & -49 & 1668 & 479 & 1050 & 21 & -402 & 153\\ -4037 & -1533 & 714 & 228 & 479 & 4390 & 126 & -87 & -153 & 0\\ 90 & -3565 & -238 & 0 & 1050 & 126 & 5100 & 0 & -668 & 96\\ 1282 & 287 & -993 & -1050 & 21 & -87 & 0 & 1336 & -96 & -592\\ 246 & -69 & -151 & -147 & -402 & -153 & -668 & -96 & 1184 & 0\\ -828 & -448 & 277 & 489 & 153 & 0 & 96 & -592 & 0 & 864 \end{array}\right)}.
$$
We can check that $\tilde Q$ is positive definite\footnote{To verify positive definiteness of $\tilde Q$ and the equality $p(x) - \big (\ell_1(x), \: \ell_2(x) \big)^T \: \nabla p(x) = \tilde z(x)^T \tilde Q \tilde z(x)$, refer to: \\ \url{http://colab.research.google.com/github/cdibek/spq_polynomials/blob/main/Theorem_3_2_Proof.ipynb}.}. Thus, $p(x) - \big (\ell_1(x), \: \ell_2(x) \big)^T \: \nabla p(x)$ is sos. This proves \eqref{eq:p1_nonnegative}.

Now, by \eqref{eq:p1_coercive} and \eqref{eq:p1_nonnegative}, it follows that $p(x)$ is nonnegative.
\end{proof}

\medskip

\begin{proof}[Proof of Theorem \ref{thm:counter_example_3_4}]
Clearly, the polynomial $p(x)$ is SPQ. As in the proof of Theorem \ref{thm:counter_example_2_6}, we show that $p(x)$ is not sos via a separating hyperplane argument. We fix the following monomial ordering:
\begin{align*}
\hspace{-0.6cm}
v =
( & 1, x_1^2, x_1x_2, x_2^2, x_1x_3, x_2x_3, x_3^2, x_1^4, x_1^3x_2, x_1^2x_2^2, x_1x_2^3, x_2^4, x_1^3x_3, x_1^2x_2x_3, x_1x_2^2x_3, x_2^3x_3, x_1^2x_3^2, x_1x_2x_3^2, \\ & x_2^2x_3^2, x_1x_3^3, x_2x_3^3, x_3^4)^T.
\end{align*}
Let the vector of coefficients of $p(x)$ in the above monomial ordering be denoted by
\begin{align*}
\hspace{-0.6cm}
\overrightarrow{p} = 
(9/4 , 2 , 8 , 2 , 8 , 8 , 2 , 1 , 0 , 0 , 0 , 1 , 0 , 0 , 0 , 0 , 0 , 0 , 0 , 0 , 0 , 1)^T.
\end{align*}
One can verify\footnote{All the computations in this proof are carried out over rational numbers and can be verified via the following link: \url{http://colab.research.google.com/github/cdibek/spq_polynomials/blob/main/Theorem_3_3_Proof.ipynb}.} that the vector 
\begin{align*}
\hspace{-0.6cm}
\mu = (120 , 95 , -47 , 95 , -47 , -47 , 95 , 95 , -47 ,  78 , -47, 95 , -47 , -10 , -10 , -47 , 78 , -10 , 78 , -47 , -47 , 95)^T
\end{align*}
satisfies $\langle \mu, \overrightarrow{p} \rangle = -3 < 0$. We claim that for any sos polynomial $q(x)$ containing only the monomials in $v$, we should have $\langle \mu, \overrightarrow{q} \rangle \geq 0$, where $\overrightarrow{q}$ denotes the coefficients of $q(x)$ listed according to the ordering in $v$. Indeed, if $q(x)$ is sos, by Theorem \ref{thm:sos_psd}, it can be written as
$$q(x) = z(x)^TQz(x) = \text{Tr} \Big(Q \: z(x)z(x)^T \Big)$$
for some positive semidefinite matrix $Q$ and vector of monomials
$$z(x) = (x_1, x_2, x_3, 1, x_1^2, x_1x_2, x_2^2, x_1x_3, x_2x_3, x_3^2)^T.$$
It is easy to see that
$$\langle \mu, \overrightarrow{q} \rangle =  \text{Tr}\left (Q \: \big (z(x)z(x)^T \big)|_\mu \right),$$
where 
$$
\big (z(x)z(x)^T \big)|_\mu =\scriptsize{
\left(\begin{array}{cccccccccc}
95 & -47 & -47 & 0 & 0 & 0 & 0 & 0 & 0 & 0 \\
-47 & 95 & -47 & 0 & 0 & 0 & 0 & 0 & 0 & 0 \\
-47 & -47 & 95 & 0 & 0 & 0 & 0 & 0 & 0 & 0 \\
0 & 0 & 0 & 120 & 95 & -47 & 95 & -47 & -47 & 95 \\
0 & 0 & 0 & 95 & 95 & -47 & 78 & -47 & -10 & 78 \\
0 & 0 & 0 & -47 & -47 & 78 & -47 & -10 & -10 & -10 \\
0 & 0 & 0 & 95 & 78 & -47 & 95 & -10 & -47 & 78 \\
0 & 0 & 0 & -47 & -47 & -10 & -10 & 78 & -10 & -47 \\
0 & 0 & 0 & -47 & -10 & -10 & -47 & -10 & 78 & -47 \\
0 & 0 & 0 & 95 & 78 & -10 & 78 & -47 & -47 & 95
\end{array}\right)}
$$
is the matrix where each monomial in $z(x)z(x)^T$ is replaced with the corresponding element of the vector $\mu$ (or zero, if the monomial is not in $v$). We can check that the matrix $\big (z(x)z(x)^T \big)|_\mu$ is positive definite. Since $Q$ and $\big (z(x)z(x)^T \big)|_\mu$ are positive semidefinite, it follows that $\langle \mu, \overrightarrow{q} \rangle = \text{Tr} \left ( Q \: \big (z(x)z(x)^T \big)|_\mu \right ) \geq 0$. This completes the proof that $p(x)$ is not sos.

Next, we prove that $p(x)$ is nonnegative by showing that the polynomial
$(x_1^2 + x_2^2 + x_3^2) \: p(x)$ is sos. Indeed, $(x_1^2 + x_2^2 + x_3^2) \: p(x) = \tilde z(x)^T \tilde Q \tilde z(x)$, where
$$\tilde z(x) =
(x_1^2, x_1x_2, x_2^2, x_1x_3, x_2x_3, x_3^2, x_1, x_2, x_3, x_1^3, x_1^2x_2, x_1x_2^2, x_2^3, x_1^2x_3, x_1x_2x_3, x_2^2x_3, x_1x_3^2, x_2x_3^2, x_3^3)^T,
$$
and $\tilde Q$ is the following matrix:
$$
\hspace{-1.2cm}
\scriptsize{\frac{1}{336}
\left(\begin{array}{ccccccccccccccccccc}
826 & 861 & 105 & 861 & 336 & 105 & 0 & 0 & 0 & 0 & 0 & 0 & 0 & 0 & 0 & 0 & 0 & 0 & 0 \\
861 & 1638 & 861 & 1239 & 1239 & 336 & 0 & 0 & 0 & 0 & 0 & 0 & 0 & 0 & 0 & 0 & 0 & 0 & 0 \\
105 & 861 & 826 & 336 & 861 & 105 & 0 & 0 & 0 & 0 & 0 & 0 & 0 & 0 & 0 & 0 & 0 & 0 & 0 \\
861 & 1239 & 336 & 1638 & 1239 & 861 & 0 & 0 & 0 & 0 & 0 & 0 & 0 & 0 & 0 & 0 & 0 & 0 & 0 \\
336 & 1239 & 861 & 1239 & 1638 & 861 & 0 & 0 & 0 & 0 & 0 & 0 & 0 & 0 & 0 & 0 & 0 & 0 & 0 \\
105 & 336 & 105 & 861 & 861 & 826 & 0 & 0 & 0 & 0 & 0 & 0 & 0 & 0 & 0 & 0 & 0 & 0 & 0 \\
0 & 0 & 0 & 0 & 0 & 0 & 756 & 0 & 0 & -77 & 294 & -126 & 189 & 294 & -21 & -105 & -126 & -105 & 189 \\
0 & 0 & 0 & 0 & 0 & 0 & 0 & 756 & 0 & 189 & -126 & 294 & -77 & -105 & -21 & 294 & -105 & -126 & 189 \\
0 & 0 & 0 & 0 & 0 & 0 & 0 & 0 & 756 & 189 & -105 & -105 & 189 & -126 & -21 & -126 & 294 & 294 & -77 \\
0 & 0 & 0 & 0 & 0 & 0 & -77 & 189 & 189 & 336 & 0 & -63 & 0 & 0 & -126 & 63 & -63 & 63 & 0 \\
0 & 0 & 0 & 0 & 0 & 0 &294 & -126 & -105 & 0 & 462 & 0 & -63 & 126 & 84 & -147 & -147 & -75 & 63 \\
0 & 0 & 0 & 0 & 0 & 0 & -126 & 294 & -105 & -63 & 0 & 462 & 0 & -147 & 84 & 126 & -75 & -147 & 63 \\
0 & 0 & 0 & 0 & 0 & 0 & 189 & -77 & 189 & 0 & -63 & 0 & 336 & 63 & -126 & 0 & 63 & -63 & 0 \\
0 & 0 & 0 & 0 & 0 & 0 & 294 & -105 & -126 & 0 & 126 & -147 & 63 & 462 & 84 & -75 & 0 & -147 & -63 \\
0 & 0 & 0 & 0 & 0 & 0 & -21 & -21 & -21 & -126 & 84 & 84 & -126 & 84 & 450 & 84 & 84 & 84 & -126 \\
0 & 0 & 0 & 0 & 0 & 0 &-105 & 294 & -126 & 63 & -147 & 126 & 0 & -75 & 84 & 462 & -147 & 0 & -63 \\
0 & 0 & 0 & 0 & 0 & 0 & -126 & -105 & 294 & -63 & -147 & -75 & 63 & 0 & 84 & -147 & 462 & 126 & 0 \\
0 & 0 & 0 & 0 & 0 & 0 & -105 & -126 & 294 & 63 & -75 & -147 & -63 & -147 & 84 & 0 & 126 & 462 & 0 \\
0 & 0 & 0 & 0 & 0 & 0 &189 & 189 & -77 & 0 & 63 & 63 & 0 & -63 & -126 & -63 & 0 & 0 & 336
\end{array}\right)}.
$$
We can check that $\tilde Q$ is positive definite\footnote{To verify positive definiteness of $\tilde Q$ and the equality $(x_1^2 + x_2^2 + x_3^2) \: p(x) = \tilde z(x)^T \tilde Q \tilde z(x)$, refer to: \\ \url{http://colab.research.google.com/github/cdibek/spq_polynomials/blob/main/Theorem_3_3_Proof.ipynb}.}, and thus $(x_1^2 + x_2^2 + x_3^2) \: p(x)$ is sos.
\end{proof}


\subsection{Examples in higher degrees and dimensions}\label{subsec:higher_n_d}

In Theorems \ref{thm:counter_example_2_6} and \ref{thm:counter_example_3_4}, we presented nonnegative SPQ polynomials that are not sos for the two minimal cases $(n, d) = (2, 6)$ and $(n, d) = (3, 4)$. In this subsection, we show how to construct SPQ polynomials that are nonnegative but not sos in higher degrees and dimensions.

\begin{theorem}
\label{thm:higher_deg_dim}
For $n \geq 2$ and $d \geq 6$, and for $n \geq 3$ and $d \geq 4$, the set $N^{SPQ}_{n,d} \setminus \Sigma^{SPQ}_{n,d}$ is nonempty.
\end{theorem}

\begin{proof}
The fact that $N^{SPQ}_{2,6} \setminus \Sigma^{SPQ}_{2,6} \neq \emptyset$ follows from Theorem \ref{thm:counter_example_2_6}. We first show existence of polynomials in $N^{SPQ}_{n,d} \setminus \Sigma^{SPQ}_{n,d}$ for $n = 2$ and $d \geq 8$. Let $p(x_1, x_2)$ be a bivariate SPQ polynomial of degree 6 that is nonnegative but not sos (e.g., the polynomial given in Theorem \ref{thm:counter_example_2_6}). Let $d$ be an even integer greater than or equal to 8. Consider the polynomial
\vspace{-0.16cm}
\begin{equation*}
\hspace{5.6cm}
q_t(x) = p(x_1, x_2) + t x_1^{d},
\end{equation*}
where $t$ is a parameter. Clearly, $q_t(x) \in N^{SPQ}_{2, d}$, $\forall t \geq 0$, and $q_0(x) \in N^{SPQ}_{2,6} \setminus \Sigma^{SPQ}_{2,6}$ since $p(x_1, x_2)$ is not sos. We claim that the polynomial $q_t(x)$ is not sos for some $t > 0$. Suppose for the sake of contradiction that $q_t(x)$ was sos for every $t > 0$. Then, by the closedness of the cone $\Sigma_{2,d}$ (see \cite{Robinson}), the polynomial $q_0(x)$ must be sos, which is a contradiction.\footnote{To make the proof completely constructive, one can solve a semidefinite program that finds the smallest $\epsilon$ which makes $p(x_1, x_2) + \epsilon x_1^{d}$ sos, and then set $t = \epsilon / 2$. (If this program is infeasible, then set $t$ to any positive real number.)}

Next, we show that from any polynomial $p(x) \in N^{SPQ}_{n,d} \setminus \Sigma^{SPQ}_{n,d}$, one can easily construct a polynomial $\tilde p(x, x_{n+1}) \in N^{SPQ}_{n+1,d} \setminus \Sigma^{SPQ}_{n+1,d}$. This, together with the result of the previous paragraph and Theorem~\ref{thm:counter_example_3_4}, would complete the proof. Consider the polynomial
\vspace{-0.12cm}
\begin{equation*}
\hspace{5.4cm}
\tilde p(x, x_{n+1}) = p(x) + x_{n+1}^{d}.
\end{equation*}
Clearly, $\tilde p(x, x_{n+1}) \in N^{SPQ}_{n+1,d}$. Moreover, $\tilde p(x, x_{n+1})$ is not a sum of squares since a decomposition $\tilde p(x, x_{n+1}) = \sum_i q_i^2(x, x_{n+1})$ would imply that $p(x) = \tilde p(x, 0) = \sum q_i^2(x, 0)$ is a sum of squares, which is a contradiction.
\end{proof}

This completes our study of the relationship between nonnegative and sos SPQ polynomials for all combinations $(n,d)$ and hence the proof of Theorem \ref{thm:main_thm}.

\vspace{-0.2cm}
\subsection{Summary of the minimal examples}

Recall the inclusion relationships
\reqnomode
\begin{equation}
\hspace{2.15cm}
(N^S+N^{Q_f})_{n,d} \quad \subseteq \quad (N^S+N^Q)_{n,d} \quad \subseteq \quad \Sigma^{SPQ}_{n,d} \quad \subseteq \quad N^{SPQ}_{n,d},
\label{eq:minimal_ex}
\end{equation}
$\hspace{6.2cm} \big \downarrow \subset \hspace{2.9cm} \big\downarrow \subset \hspace{1.5cm} \big \downarrow \subset$

\vspace{0.2cm}

$\hspace{5.3cm} (2,2) \hspace{2.8cm} (2,4) \hspace{0.9cm} (2,6), (3,4)$

\vspace{0.2cm}

\noindent where the downward arrows point to minimal $(n, d)$ for which the inclusion is strict, as proven in Lemma \ref{lem:minimal_2_2}, Lemma \ref{lem:counterexample_for_nonneg_sum}, Theorem \ref{thm:counter_example_2_6}, and Theorem~\ref{thm:counter_example_3_4}, respectively. We summarize the examples showing the strictness of these inclusions in Table \ref{tab:minimal_ex}.

\begin{table}[h]
\centering
\begin{tabular}{|| m{11.8cm} | m{4.8cm} ||} 
 \hline
$p(x)$ &  \vspace{0.1cm} $p(x)$ belongs to  \\ [0.3ex] 
 \hline\hline
$p(x) = 8x_1^2 - 4x_1 + 2x_2^2 - 2x_2 + 3 + 8x_1x_2$ & \vspace{0.11cm} $(N^S+N^Q)_{2,2} \setminus (N^S+N^{Q_f})_{2,2}$  \\  [1.3ex] 
 \hline
$p(x) = x_1^4 - x_1^2 + 2x_1 + x_2^4 - x_2^2 - 2x_2 + \frac{12}{5} - 2x_1x_2$ & \vspace{0.11cm} $\Sigma^{SPQ}_{2, 4} \setminus (N^S+N^Q)_{2, 4}$  \\ [1.3ex] 
 \hline
$p(x) = 17x_1^6 - 20x_1^4 + 7x_1^2 + 18x_1 + 18x_2^4 - 19x_2^2 - 19x_2 + 21 - 20x_1x_2$ & \vspace{0.11cm} $N^{SPQ}_{2,6} \setminus \Sigma^{SPQ}_{2,6}$ \\ [1.3ex] 
 \hline
$p(x) = x_1^4 + 2x_1^2 + x_2^4 + 2x_2^2 + x_3^4 + 2x_3^2 + \frac{9}{4} + 8x_1x_2 + 8x_1x_3 + 8x_2x_3$ & \vspace{0.11cm} $N^{SPQ}_{3,4} \setminus \Sigma^{SPQ}_{3,4}$ \\ [1.3ex] 
 \hline
\end{tabular}
\vspace{-0.2cm}
\caption{Minimal examples showing the strictness of the inclusions in \eqref{eq:minimal_ex}}
\label{tab:minimal_ex}
\end{table}

\section{Complexity of Deciding Nonnegativity of SPQ Polynomials} \label{sec:complexity}

Theorem \ref{thm:main_thm} indicates that nonnegativity of an SPQ polynomial cannot always be established by a decomposition as a sum of squares of polynomials. One might still wonder whether there exists a characterization of nonnegative SPQ polynomials which can be checked in polynomial time. In this section, we show that, unless P$=$NP, the answer to this question is negative. Our hardness result holds already for SPQ polynomials of degree 4.

\begin{theorem}
The following decision problem is NP-hard:
\normalfont
\decision{\textsc{SPQ-Nonnegativity}}
{Rational coefficients of a quartic SPQ polynomial $p(x)$ given in a fixed monomial ordering.}
{Decide if $p(x) \geq 0$ for all $x \in \R^n$.}
\label{thm:NP_hard_nonneg}
\end{theorem}

\begin{proof}
We prove the claim by giving a polynomial-time reduction from the \textsc{Partition} problem, which is known to be NP-hard~\cite{GarJohn}. Recall that in \textsc{Partition}, we are given a set of positive integers and asked whether it is possible to split them into two subsets with equal sums:
\decision{\textsc{Partition}}
{A set of positive integers $\{a_1, \dots, a_n\}$.}
{Decide if there exists a partition of $\{1,2,\dots,n\}$ into two subsets $I$ and $I^C$ such that $\sum\limits_{i \in I} a_i= \sum\limits_{j \in I^C} a_j$.}
We say that a \textsc{Partition} instance is \emph{feasible} (resp. \emph{infeasible}) if the answer to the \textsc{Partition} question is yes (resp. no). Consider a given \textsc{Partition} instance $\{a_1, \dots, a_n\}$. Let 
$N = 18 \Big ( \max\limits_i a_i \Big ) \Big ( \sum\limits_{i} a_i \Big )^3$, and $c = \frac{1}{2}$. Define the quartic SPQ polynomial
$$p(x) = N \sum\limits_{i=1}^{n} (x_i^2 - 1)^2 + \Big ( \sum\limits_{i=1}^{n} a_ix_i\Big )^2 - c.$$
We show that the \textsc{Partition} instance $\{a_1, \dots, a_n\}$ is infeasible if and only if $p(x)$ is nonnegative. This would complete the proof.

Assume first that $\{a_1, \dots, a_n\}$ is feasible. Let $x_i^* = 1$ if $i \in I$, and $x_i^* = -1$ if $i \in I^C$. Then, $\sum_{i=1}^{n} a_ix_i^* = 0$ since $\{a_1, \dots, a_n\}$ is feasible, and so $p(x^*) = - c < 0$. Therefore, $p(x)$ is not nonnegative.

\reqnomode
Assume now that $\{a_1, \dots, a_n\}$ is infeasible. To show that $p(x)$ is nonnegative, we prove that if $x^*$ is a global minimum of $p(x)$, then $p(x^*) \geq 0$. Notice that since the top homogeneous component of $p(x)$ equals $N \sum_{i=1}^{n} x_i^4$, the polynomial $p(x)$ is coercive, and so it achieves its infimum. We will make use of the following partial derivatives of $p(x)$:
$$\frac{\partial p}{\partial x_i} (x) = 4N (x_i^2 -1)x_i + 2 \Big (\sum\limits_{j=1}^{n} a_jx_j \Big ) a_i \quad \text{for} \quad i=1, \dots, n,$$
$$\frac{\partial^2 p}{\partial x_i^2} (x) = 4N (3x_i^2 - 1) + 2 a_i^2 \quad \text{for} \quad i=1, \dots, n.$$
We see from the second-order necessary condition for optimality that $x = 0$ is not a local minimum of $p(x)$ since, e.g., $\frac{\partial^2 p}{\partial x_1^2} (0) = -4N + 2 a_1^2 < 0$.
Let $x^* \neq 0$ be a global minimum of $p(x)$. The first-order necessary condition for optimality implies that $\nabla p(x^*) = 0$, i.e.,
$$
4N (x_i^{*^2} -1)x_i^* = - 2(\sum\limits_{j=1}^{n} a_jx_j^*) a_i \:\:\: \text{for} \:\:\: i=1,\dots,n.
$$
Taking the absolute value of both sides, we obtain
\begin{equation}
\hspace{3.7cm}
2N |x_i^{*^2} -1| |x_i^*| = \Big | \sum\limits_{j=1}^{n} a_jx_j^* \Big | a_i \:\:\: \text{for} \:\:\: i=1,\dots,n.
\label{eq:first_order}
\end{equation}
Next, we prove three claims about the entries of $x^*$, with the third claim enabling us to show that $p(x^*) \geq 0$.

\vspace{-0.3cm}

\leqnomode
\begin{equation}
\longbox{{\it For $i=1,\dots, n$, we have $\frac{1}{6} \leq |x_i^*| \leq \frac{3}{2}$.}}
\label{eq:sandwich}\tag{$4.1.1$}
\end{equation}
Let $i_+ \in \argmax\limits_{i \in \{1,\dots,n\}} |x_i^*|$. As $x^* \neq 0$, $|x_{i_+}^*| > 0$. From \eqref{eq:first_order}, we have
$$2N |x_{i_+}^{*^2} -1| |x_{i_+}^*| = a_{i_+}  \Big | \sum\limits_{j=1}^{n} a_jx_j^* \Big | \leq a_{i_+}  \sum\limits_{j=1}^{n} a_j |x_j^*| \leq a_{i_+}  \sum\limits_{j=1}^{n} a_j |x_{i_+}^*|.$$
Hence, $2N |x_{i_+}^{*^2} -1| \leq a_{i_+} \sum\limits_{j=1}^{n} a_j$, and so $$|x_{i_+}^{*^2} -1| \leq \frac{a_{i_+} \sum\limits_{j=1}^{n} a_j}{2N} \leq \frac{1}{2}.$$ 
It follows that $x_{i_+}^{*^2} \leq \frac{3}{2}$, and thus $x_{i}^{*^2} \leq \frac{3}{2}$ for $i = 1, \dots, n$. To obtain the desired upper bound on $|x_i^*|$, observe that $|x_i^*| \leq \sqrt{\frac{3}{2}} \leq \frac{3}{2}$.

For the lower bound, we make use of the second-order necessary condition for optimality, which implies that
$$\frac{\partial^2 p}{\partial x_i^2} (x^*) = 4N (3x_i^{*^2} - 1) + 2 a_i^2 \geq 0 \quad \text{for} \quad i = 1, \dots, n.$$
This yields $x_i^{*^2} \geq \frac{1}{3} - \frac{a_i^2}{6N} \geq \frac{1}{6}$, where the second inequality follows since $\frac{a_i^2}{6N} \leq \frac{1}{6}$. Thus, $x_i^{*^2} \geq \frac{1}{6}$ for $i = 1, \dots, n$, and $|x_i^*| \geq \sqrt{\frac{1}{6}} \geq \frac{1}{6}$.

\vspace{-0.3cm}

\begin{equation}
\mbox{{\it For $i=1,\dots, n$, we have $|x_i^{*^2} - 1| \leq \frac{1}{4} \: \frac{1}{(\sum\limits_j a_j )^2}$.}}
\label{eq:upper_bound}\tag{$4.1.2$}
\end{equation}
Observe that
$$
\frac{1}{6} 2N |x_i^{*^2} -1| \leq 2N |x_i^{*^2} -1| |x_i^*| = a_i \Big | \sum\limits_{j} a_jx_j^* \Big | \leq a_i \sum\limits_{j} a_j |x_j^*| \leq \frac{3}{2} a_i \sum\limits_{j} a_j,
$$
where the first and the last inequalities follow from \eqref{eq:sandwich} and the equality follows from \eqref{eq:first_order}. Hence, for $i=1,\dots, n$, we obtain
$$|x_i^{*^2} - 1| \leq \frac{9}{2} \:\: \frac{a_i \sum\limits_{j} a_j}{N} = \frac{9}{2} \:\: \frac{a_i \sum\limits_{j} a_j}{18 \big ( \max\limits_j a_j \big ) \big ( \sum\limits_{j} a_j \big )^3} \leq \frac{1}{4} \: \frac{1}{(\sum\limits_j a_j)^2}.$$

\vspace{-0.3cm}

\newpage
\begin{equation}
\mbox{{\it For $i=1,\dots,n$, $x_i^* = \sgn(x_i^*) + \epsilon_i$, where $|\epsilon_i| \leq \frac{1}{4} \: \frac{1}{(\sum\limits_i a_i)^2}$ and $\sgn(\cdot)$ is the sign function.}}
\label{eq:sgn_description}\tag{$4.1.3$}
\end{equation}
For ease of notation, let $m = \frac{1}{4} \: \frac{1}{(\sum\limits_j a_j)^2}$. From \eqref{eq:upper_bound}, we have $1 - m \leq x_i^{*^2} \leq 1 + m$ for $i=1,\dots,n$. Since $0 < 1 - m < 1$ and $1 < 1 + m < 2$, it follows that
$$1 - m \leq \sqrt{1 - m} \leq |x_i^*| \leq \sqrt{1 + m} \leq 1 + m,$$
and thus $1 - m \leq |x_i^*| \leq 1 + m$.
If $\sgn(x_i^*) > 0$, then
$$1 - m \leq |x_i^*| \leq 1 + m \implies \sgn(x_i^*) - m \leq x_i^* \leq \sgn(x_i^*) + m,$$
and if $\sgn(x_i^*) < 0$, then
$$1 - m \leq |x_i^*| \leq 1 + m \implies \sgn(x_i^*) + m \geq x_i^* \geq \sgn(x_i^*) - m.$$
Thus, in both cases, we obtain $-m \leq x_i^* - \sgn(x_i^*) \leq m$, that is, $|x_i^* - \sgn(x_i^*)| \leq m$. Finally, letting $\epsilon_i = x_i^* - \sgn(x_i^*)$, it follows that $x_i^* = \sgn(x_i^*) + \epsilon_i$ where $|\epsilon_i| \leq \frac{1}{4} \: \frac{1}{(\sum\limits_i a_i)^2}$. This proves \eqref{eq:sgn_description}.

\medskip

\noindent We finish the proof by showing that $p(x^*) \geq 0$. Define
$$\text{min gap} = \min_{I \subseteq \{1, \dots, n\}} \Bigg | \sum\limits_{i \in I} a_i - \sum\limits_{j \in I^C} a_j \Bigg |.$$
We have
\reqnomode
\begin{align}
\hspace{3cm}
p(x^*) & = N \sum\limits_{i=1}^{n} (x_i^{*^2} - 1)^2 + \Big ( \sum\limits_{i=1}^{n} a_ix_i^*\Big )^2 - c \nonumber \\
& \geq \Big ( \sum\limits_{i=1}^{n} a_ix_i^*\Big )^2 - c \label{eq:a}\\
& = \Big ( \sum\limits_{i=1}^{n} a_i \sgn(x_i^*) + \sum\limits_{i=1}^{n} a_i \epsilon_i \Big )^2 - c \label{eq:b}\\
& \geq \Big ( \sum\limits_{i=1}^{n} a_i \sgn(x_i^*) \Big )^2 - 2 \Big |\sum\limits_{i=1}^{n} a_i \sgn(x_i^*) \Big | \Big |\sum\limits_{i=1}^{n} a_i \epsilon_i \Big | - c \label{eq:c}\\
& \geq ( \text{min gap} )^2 - 2 \Big |\sum\limits_{i=1}^{n} a_i \sgn(x_i^*) \Big | \Big |\sum\limits_{i=1}^{n} a_i \epsilon_i \Big | - c \label{eq:d},
\end{align}
where \eqref{eq:a} follows since $N \sum_{i=1}^{n} (x_i^{*^2} - 1)^2 \geq 0$, the equality \eqref{eq:b} follows by \eqref{eq:sgn_description}, the inequality \eqref{eq:c} follows since $(a + b)^2 \geq a^2 - 2 |a||b|$, and the inequality \eqref{eq:d} follows since $\Big ( \sum\limits_{i=1}^{n} a_i \sgn(x_i^*) \Big )^2 \geq ( \text{min gap} )^2$.

\medskip

\noindent Now, since $\Big |\sum\limits_{i} a_i \sgn(x_i^*) \Big | \leq \sum\limits_{i} a_i$ and 
$\Big |\sum\limits_{i} a_i \epsilon_i \Big | \leq \sum\limits_{i} a_i |\epsilon_i| \leq \frac{1}{4} \: \frac{1}{(\sum\limits_i a_i)^2} \sum\limits_{i} a_i = \frac{1}{4} \: \frac{1}{\sum\limits_i a_i}$, we have
$$\Big |\sum\limits_{i} a_i \sgn(x_i^*) \Big | \: \Big |\sum\limits_{i} a_i \epsilon_i \Big | \leq \frac{1}{4}.$$
Moreover, since $a_i$'s are integers and $\{a_1, \dots, a_n\}$ is infeasible, we have min gap $\geq 1$. Hence, we obtain
$$p(x^*) \geq ( \text{min gap} )^2 - 2 \Big |\sum\limits_{i=1}^{n} a_i \sgn(x_i^*) \Big | \Big |\sum\limits_{i=1}^{n} a_i \epsilon_i \Big | - c \geq 1 - 2 \: \frac{1}{4} - \frac{1}{2} = 0.\qedhere$$
\end{proof}

\section{Convex SPQ Polynomials}
\label{sec:convexity}
\reqnomode

In this section, we study the problems of checking convexity of SPQ polynomials (Section \ref{subsec:check_convexity}), checking nonnegativity of convex SPQ polynomials (Section \ref{subsec:cone_nonneg_sum}), and solving polynomial optimization problems where the objective and the constraints are given by convex SPQ polynomials (Section \ref{subsec:spq_optimization}).

We start by recalling some definitions. We refer to a matrix with polynomial entries as a polynomial matrix. A symmetric $m \times m$ polynomial matrix $P(x)$ in $n$ variables is said to be positive semidefinite, denoted by $P(x) \succeq 0$, if $P(x)$ is positive semidefinite for all $x \in \mathbb{R}^n$. It is straightforward to see that this condition holds if and only if the (scalar) polynomial $y^T P(x)y$ in $m+n$ variables $(x, y)$ is nonnegative. A polynomial matrix $P(x)$ is said to be an sos-matrix if $P(x) = A(x)^TA(x)$ for some polynomial matrix $A(x)$; equivalently, if the polynomial $y^T P(x)y$ in the variables $(x, y)$ is sos. Obviously, if $P(x)$ is an sos-matrix, then $P(x)$ is positive semidefinite.

We recall that a polynomial $p(x)$ is convex if and only if its Hessian $H(x)$, i.e., the $n \times n$ symmetric matrix of its second derivatives, is positive semidefinite. It is known that the problem of deciding convexity of polynomials is strongly NP-hard already for polynomials of degree four~\cite{AOPT}. Helton and Nie~\cite{HeltonNie} proposed the notion of \emph{sos-convexity} as a tractable algebraic certificate for convexity of polynomials. We say that a polynomial $p(x)$ is sos-convex if its Hessian $H(x)$ is an sos-matrix. Sos-convexity is obviously a sufficient condition for convexity. Moreover, as a consequence of Theorem \ref{thm:sos_psd}, the problem of deciding if a given polynomial is sos-convex amounts to solving a single semidefinite program.

\subsection{Deciding convexity of SPQ polynomials}\label{subsec:check_convexity}

It is known that polynomials of degree 4 or larger can be convex without being sos-convex (see, e.g.,~\cite{AhmPar1, AhmPar2}). An implication of our next theorem is that this cannot happen for SPQ polynomials, i.e., every convex SPQ polynomial is sos-convex. Hence, deciding convexity of an SPQ polynomial can be reduced to solving a semidefinite program of tractable size. In fact, the proof of the theorem below shows that convexity of an $n$-variate SPQ polynomial can be decided by simply finding the roots of $n$ univariate polynomials and checking whether a constant matrix is positive semidefinite. This is in contrast to the NP-hardness result (Theorem \ref{thm:NP_hard_nonneg}) for deciding nonnegativity of an SPQ polynomial.

\begin{theorem}
The Hessian of a convex SPQ polynomial can be written as the sum of positive semidefinite univariate polynomial matrices. In particular, an SPQ polynomial is convex if and only if it is sos-convex.
\label{thm:convex_spq_sos_convex}
\end{theorem}

\begin{proof}
Let $p(x) = q(x) + \sum_{i=1}^n u_i(x_i)$ be a convex SPQ polynomial, where $q(x)$ is a quadratic polynomial and $u_i(x_i)$, for $i=1,\dots,n$, is a univariate polynomial in $x_i$. Observe that the Hessian $H(x)$ of $p(x)$ is given as $$H(x) = C + U(x),$$
where $C$ is a constant matrix (which corresponds to the Hessian of $q(x)$) and $U(x)$ is an $n \times n$ diagonal matrix with the vector
$\left (u_1''(x_1), \dots, u_n''(x_n) \right)^T$
on its diagonal. We may assume that the matrix $C$ has zero diagonal since, for $i=1,\dots,n$, we may push, if necessary, the $x_i^2$ term of $q(x)$ into $u_i(x_i)$.

Since $p(x)$ is convex, $H(x) \succeq 0$. Hence, for $i=1, \dots, n$, the univariate polynomial $u_i''(x_i)$ is nonnegative. Let $x_i^* \in \R$ be a global minimum of $u_i''(x_i)$ and let $D$ be the $n \times n$ diagonal matrix with the vector $(u_1''(x_1^*), \dots, u_n''(x_n^*))^T$ on its diagonal. Since $H(x_1^*, \dots, x_n^*) \succeq 0$, it follows that $C+D \succeq 0$. Moreover, $U(x) - D$ is a diagonal matrix with nonnegative diagonal entries $u_i''(x_i) - u_i''(x_i^*)$, and therefore positive semidefinite. Thus, $H(x)$ can be written as
$$H(x) = H^0 + \sum_{i=1}^n H^i(x_i),$$
where $H^0 = C+D$ and $H^i(x_i)$ is an $n \times n$ matrix with $H_{ii}^i(x_i) = u_i''(x_i) - u_i''(x_i^*)$ and zero everywhere else. Since $H^0$ and $H^i(x_i)$, $i=1,\dots,n$, are positive semidefinite, this shows that $H(x)$ can be written as the sum of positive semidefinite univariate polynomial matrices. To see that $p(x)$ is sos-convex, observe that the scalar polynomial $y^TH(x)y = y^TH^0y + \sum_{i=1}^n y_i^2 (u_i''(x_i) - u_i''(x_i^*))$ is sos, and therefore $H(x)$ is an sos-matrix.
\end{proof}

\subsection{Nonnegativity of convex SPQ polynomials}\label{subsec:cone_nonneg_sum}

It is known that there are nonnegative convex polynomials that are not sos \cite{Blekherman2}, though the construction of an explicit example remains an open problem. In \cite{HeltonNie}, it is shown that a nonnegative sos-convex polynomial is sos. This result, combined with Theorem \ref{thm:convex_spq_sos_convex}, immediately gives the following corollary.

\begin{corollary}
A nonnegative convex SPQ polynomial is sos.
\label{cor:convex_nonneg_sos}
\end{corollary}
The main result of this subsection is the following stronger statement, which also has computational implications for convex SPQ polynomial optimization (see Section \ref{subsec:spq_optimization}).

\leqnomode
\begin{theorem}
A nonnegative convex SPQ polynomial can be written as the sum of a nonnegative separable and a nonnegative quadratic polynomial.
\label{thm:convex_nonnegative_sum}
\end{theorem}

\begin{proof}
Let $p(x) = q(x) + \sum_{i=1}^n u_i(x_i)$ be a nonnegative convex SPQ polynomial, where $q(x)$ is a quadratic polynomial and $u_i(x_i)$, for $i=1,\dots,n$, is a univariate polynomial in $x_i$. Since $p(x)$ is a lower-bounded convex polynomial, its infimum is attained (see, e.g., \cite{BelKlat}). Let $\bar x \in \R^n$ be a global minimum of $p(x)$, and note that $p(\bar x) \geq 0$ as $p(x)$ is nonnegative. In fact, we may assume that $p(\bar x) = 0$. Indeed, if we define $g(x) = p(x) - p(\bar x)$, then $g(x)$ is a nonnegative convex SPQ polynomial and satisfies $g(\bar x) = 0$. Moreover, if $g(x) \in (N^S+N^Q)_{n,d}$, then $p(x) \in (N^S+N^Q)_{n,d}$ as $p(\bar x) \geq 0$. Hence, from now on, we assume that $p(\bar x) = 0$. We also know that $\nabla p(\bar x) = 0$ by the first order optimality condition.

Since $p(x)$ is convex, as in the proof of Theorem \ref{thm:convex_spq_sos_convex}, the Hessian $H(x)$ of $p(x)$ can be written as $$H(x) = C + D + \text{Diag}(u_1''(x_1) - d_1, \dots, u_n''(x_n) - d_n),$$ where
$C$ is the Hessian of $q(x)$, $d_i \in \R$, for $i=1, \dots, n$, is the minimum value of $u_i''(x_i)$, $D$ is the diagonal matrix with $(d_1, \dots, d_n)^T$ on its diagonal, and $\text{Diag}(u_1''(x_1) - d_1, \dots, u_n''(x_n) - d_n)$ is the diagonal matrix with $(u_1''(x_1) - d_1, \dots, u_n''(x_n) - d_n)^T$ on its diagonal. We know from the proof of Theorem \ref{thm:convex_spq_sos_convex} that $C+D$ is positive semidefinite. We now prove two claims.

\vspace{-0.4cm}

\begin{equation}
\mbox{{\it We have $p(x) = \frac{1}{2} (x - \bar x)^T C (x - \bar x) + \mathlarger{\sum\limits_{i=1}^n} \Big (u_i(x_i) - u_i(\bar x_i) - u_i'(\bar x_i) (x_i - \bar x_i) \Big) $.}}
\label{eq:p_x_Taylor}\tag{$5.3.1$}
\end{equation}
Since $q(x)$ is a quadratic polynomial, by Taylor expansion around $\bar x$, we have
$$q(x) = q(\bar x) + \nabla q(\bar x)^T (x - \bar x) + \frac{1}{2} (x - \bar x)^T C (x - \bar x).$$
Similarly, by Taylor expansion of $u_i(x_i)$ around $\bar x_i$, we have
$$u_i(x_i) = u_i(\bar x_i) + u_i'(\bar x_i) (x_i - \bar x_i) + R_i(x_i),$$
where $R_i(x_i)$ is the remainder polynomial in the expansion. Then,
$$p(x) = q(x) + \sum_{i=1}^n u_i(x_i) = p(\bar x) +  \nabla p(\bar x)^T (x - \bar x) + \frac{1}{2} (x - \bar x)^T C (x - \bar x) + \sum_{i=1}^n R_i(x_i).$$
Since, $p(\bar x) = 0$ and $\nabla p(\bar x) = 0$, and since $R_i(x_i) = u_i(x_i) - u_i(\bar x_i) - u_i'(\bar x_i) (x_i - \bar x_i)$, the claim follows.

\vspace{-0.3cm}

\begin{equation}
\mbox{{\it For $i=1, \dots, n$, we have $u_i(x_i) - u_i(\bar x_i) - u_i'(\bar x_i) (x_i - \bar x_i) - \frac{1}{2} (x_i - \bar x_i)^2 d_i \geq 0$.}}
\label{eq:R_i_minus_d_i}\tag{$5.3.2$}
\end{equation}
Since $u_i''(x_i) - d_i \geq 0$, it follows that the univariate polynomial $f_i(x_i) = u_i(x_i) - \frac{1}{2} (x_i - \bar x_i)^2 d_i$ is convex. By the first-order characterization of convexity, we obtain 
$$u_i(x_i) - \frac{1}{2} (x_i - \bar x_i)^2 d_i = f_i(x_i) \geq f_i(\bar x_i) + f_i'(\bar x_i) (x_i - \bar x_i) = u_i(\bar x_i) + u_i'(\bar x_i) (x_i - \bar x_i),$$
which proves \eqref{eq:R_i_minus_d_i}. 

Now, observe that the equation in \eqref{eq:p_x_Taylor} can be rewritten as
$$p(x) = \frac{1}{2} (x - \bar x)^T (C+D) (x - \bar x) + \mathlarger{\sum\limits_{i=1}^n} \Big (u_i(x_i) - u_i(\bar x_i) - u_i'(\bar x_i) (x_i - \bar x_i) - \frac{1}{2} (x_i - \bar x_i)^2 d_i \Big).$$
The first term in this sum is a nonnegative quadratic polynomial since $C+D$ is positive semidefinite. Also, by \eqref{eq:R_i_minus_d_i}, each summand in the second term is nonnegative, and therefore the second term is a nonnegative separable polynomial.
\end{proof}

\begin{remark}
By contrasting Lemma \ref{lem:counterexample_for_nonneg_sum} and Theorem~\ref{thm:convex_nonnegative_sum}, we see that sos SPQ polynomials cannot always be written as the sum of a nonnegative separable and a nonnegative quadratic polynomial, while nonnegative convex SPQ polynomials can.
\end{remark}

\subsection{Convex SPQ polynomial optimization}\label{subsec:spq_optimization}
\reqnomode

In this subsection, we focus on polynomial optimization problems where the objective and the constraint set are defined by convex SPQ polynomials, and study some implications of Theorem \ref{thm:convex_nonnegative_sum} for such optimization problems.

\subsubsection{Unconstrained case}
\label{subsubsec:unconstrained}

For a convex polynomial $p(x)$ in $n$ variables and degree $2d \geq 2$, consider the problem of finding $p^{*} \mathrel{\mathop:}= \inf\limits_{x \in \R^n} p(x)$, and recovering an optimal solution (whose existence is guaranteed \cite{BelKlat}). Observe that
\begin{equation}
\hspace{5cm}
\begin{aligned}
p^{*} \:\:\: = \:\:\: 
& \sup\limits_{\gamma \in \R}
& & \gamma \\
& \text{s.t.}
&& p(x) - \gamma \in N_{n,2d}.
\end{aligned}
\label{eq:min_p_nonneg}
\end{equation}
A semidefinite programming-based lower bound $p^{sos}$ on $p^*$ can be obtained by replacing the nonnegativity constraint in \eqref{eq:min_p_nonneg} with an sos constraint:
\begin{equation}
\hspace{4.9cm}
\begin{aligned}
p^{sos} \:\:\: \mathrel{\mathop:}= \:\:\:
& \sup\limits_{\gamma \in \R}
& & \gamma \\
& \text{s.t.}
&& p(x) - \gamma \in \Sigma_{n,2d}.
\end{aligned}
\label{eq:min_p_sos}
\end{equation}
Since $p(x)$ is convex, the polynomial $p(x) - \gamma$ is convex for any scalar $\gamma$. Therefore, if it were true that nonnegative convex polynomials are sos, then we would have $p^* = p^{sos}$. While we know that this is not true in general \cite{Blekherman2}, Corollary \ref{cor:convex_nonneg_sos} implies that $p^* = p^{sos}$ if $p(x)$ is also SPQ. In fact, by Theorem~\ref{thm:convex_nonnegative_sum}, the following stronger statement holds when $p(x)$ is convex and SPQ:
\begin{equation}
\hspace{5cm}
\begin{aligned}
p^{*} \:\:\: = \:\:\:
& \sup\limits_{\gamma \in \R}
& & \gamma \\
& \text{s.t.}
&& p(x) - \gamma \in (N^S+N^Q)_{n,2d}.
\end{aligned}
\label{eq:min_p_sos_SPQ}
\end{equation}
While the bounds obtained from the semidefinite programs \eqref{eq:min_p_sos} and \eqref{eq:min_p_sos_SPQ} are the same for convex SPQ polynomials, the size of the semidefinite program \eqref{eq:min_p_sos_SPQ} can be significantly smaller. Indeed, suppose that $p(x)$ is a convex SPQ polynomial in $n$ variables and degree $2d$. Then, the semidefinite constraint in~\eqref{eq:min_p_sos} (cf. Theorem \ref{thm:sos_psd}) is of size $\binom{n+d}{d} \times \binom{n+d}{d}$. By contrast, to implement \eqref{eq:min_p_sos_SPQ}, we impose that $p(x) - \gamma = q(x) + \sum_{i=1}^n u_i(x_i)$, and require the quadratic polynomial $q(x)$ and each univariate polynomial $u_i(x_i)$ to be sos. This reduces the size of the largest semidefinite constraint to $\max \Big \{ \binom{n+1}{1} \times \binom{n+1}{1}, \: \binom{d+1}{d} \times \binom{d+1}{d} \Big \}$. To demonstrate how significant the difference can be, we compare in Table \ref{tab:runtime_comparison} the running times of the two approaches on randomly generated convex SPQ polynomials of different dimensions and degrees. While both approaches always return the same (tight) bound as expected, the difference in running times can be observed even at low degrees and dimensions. All experiments were done using MATLAB, the solver MOSEK \cite{MOSEK}, and a computer with 2.6 GHz speed and 8 GB RAM.

\begin{table}[h]
\centering
\begin{tabular}{|| V{2.5cm} | V{1.1cm} | V{1.1cm} | V{1.1cm} | V{1.1cm} | V{1.1cm} | V{1.1cm} | V{1.1cm} | V{1.1cm} | V{1.1cm} ||} 
 \hline
$(n,2d)$ & (4,10) & (4,12) & (4,14) & (5,10) & (5,12) & (5,14) & (6,10) & (6,12) & (6,14) \tabularnewline [1ex] 
 \hline\hline
\vspace{0.1cm} $\Sigma_{n,2d}$ & 7.44 & 22.43 & 73.86 & 33.72 & 184.92 & 2206.3 & 168.84 & 1894.4 & NA \tabularnewline [1.4ex]
 \hline
\vspace{0.1cm} $(N^S+N^Q)_{n,2d}$ & 5.82 & 6.28 & 5.29 & 6.99 & 4.77 & 4.46 & 4.86 & 4.76 & 5.62 \tabularnewline  [1.4ex]
 \hline
\end{tabular}
\caption{Comparison of running times (in seconds) of the semidefinite programs \eqref{eq:min_p_sos} and \eqref{eq:min_p_sos_SPQ} for minimizing a convex SPQ polynomial}
\label{tab:runtime_comparison}
\end{table}

\vspace{-0.4cm}

\paragraph{Recovering an optimal solution.}
\label{par:recover_opt_pt}
The following proposition shows that Theorem~\ref{thm:convex_nonnegative_sum} can be further exploited to recover a minimizer of a convex SPQ polynomial.

\begin{proposition}
Let $p(x)$ be a convex SPQ polynomial in $n$ variables and degree $2d \geq 2$, and define $p^{*} \mathrel{\mathop:}= \inf_{x \in \R^n} p(x)$. Then, a point $x^* \in \R^n$ such that $p(x^*) = p^*$ can be recovered by finding zeros of at most $n$ nonnegative univariate polynomials and at most one nonnegative quadratic polynomial, which can be obtained by solving a semidefinite program.
\label{prop:unconstrained_recover}
\end{proposition}

\begin{proof}
Since $p(x) - p^*$ is a nonnegative convex SPQ polynomial, by Theorem \ref{thm:convex_nonnegative_sum}, it can be written as $p(x) - p^* = q(x) + \sum_{i=1}^n u_i(x_i)$, where $q(x)$ is a nonnegative quadratic polynomial (and therefore \emph{convex}), and for $i =1,\dots,n$, $u_i(x_i)$ is a nonnegative \emph{convex} univariate polynomial. (Indeed, it can be easily verified that the univariate polynomials obtained at the end of the proof of Theorem \ref{thm:convex_nonnegative_sum} are convex.) To obtain $p^*$ and such a decomposition of $p(x)-p^*$, we solve the semidefinite program \eqref{eq:min_p_sos_SPQ} with  additional convexity constraints on the univariate polynomials of the separable part.\footnote{Note that a convexity constraint on a univariate polynomial is a semidefinite constraint since it is equivalent to requiring the second derivative of the polynomial to be sos.}

Let $\bar x=(\bar x_1,\dots, \bar x_n)^T$ be a minimizer of $p(x)$. Since the polynomials $q(x), u_1(x_1), \dots, u_n(x_n)$ are nonnegative, the equality $p(\bar x) - p^* = 0$ implies that $q(\bar x) = 0$ and $u_i(\bar x_i) = 0$ for $i =1,\dots,n$. Note that unless $u_i(x_i)$ is identically zero, as a univariate convex polynomial, it has a unique minimizer. Without loss of generality, for some $k \leq n$, let $u_1(x_1), \dots, u_k(x_k)$ be the univariate polynomials that are not identically zero. For $i=1,\dots,k$, the unique minimizer $x_i^*$ of $u_i(x_i)$ can be obtained by setting $u_i(x_i) = 0$. Then, we can let $x_{k+1}^*, \dots, x_n^*$ be any solution to $q(x_1^*, \dots, x_k^*, x_{k+1}, \dots, x_n) = 0$. Now, the point $x^* = (x_1^*, \dots, x_k^*, x_{k+1}^*, \dots, x_n^*)^T$ satisfies $p(x^*) = p^*$.
\end{proof}

\subsubsection{Constrained case}
\reqnomode

A \emph{polynomial optimization problem (POP)} is a problem of the form
\begin{equation}
\hspace{6cm}
\begin{aligned}
& \inf\limits_{x \in \R^n}
& &p(x) \\
& \text{s.t.}
&& x \in K,
\end{aligned}
\label{eq:pop}
\end{equation}
where $K \mathrel{\mathop:}= \{x \in \R^n \: | \: g_i(x) \leq 0, i = 1, \dots, m\}$, and $p(x), g_1(x), \dots, g_m(x)$ are polynomial functions. It is straightforward to see that 
the optimal value $p^{*}$ of problem \eqref{eq:pop} is equal to
\begin{equation}
\hspace{4.6cm}
\begin{aligned}
p^{*} \:\:\: = \:\:\:
& \sup\limits_{\gamma \in \R}
& & \gamma \\
& \text{s.t.}
&& p(x) - \gamma \geq 0, \quad \forall x \in K.
\end{aligned}
\label{eq:min_p_nonneg_const}
\end{equation}
In \cite{Lasserre2}, Lasserre introduced a hierarchy of semidefinite programming-based lower bounds on $p^*$ that asymptotically converge to $p^*$ under a certain assumption on $K$. More recently, finite convergence of this hierarchy has been studied for convex polynomial optimization problems \cite{Lasserre, Klerk}. In particular, under the Slater regularity assumption, the Lasserre hierarchy is known to converge in one step if $p(x), g_1(x),\dots, g_m(x)$ are sos-convex \cite{Lasserre}. Hence, if $p(x), g_1(x),\dots, g_m(x)$ are convex SPQ polynomials, the previous statement, together with Theorem \ref{thm:convex_spq_sos_convex}, implies that $p^*$ can be found by solving a single semidefinite program (associated with the first level of the Lasserre hierarchy). We next show that in the same setting, we can find $p^*$ by solving a much smaller semidefinite program. Moreover, under a mild additional assumption on $p(x)$, we present a procedure to recover an optimal solution to \eqref{eq:pop}.

\begin{theorem}
Consider the polynomial optimization problem \eqref{eq:pop} and assume that the Slater regularity condition holds for $g_1(x), \dots, g_m(x)$, i.e., there exists $\bar x \in \R^n$ such that $g_i(\bar x) < 0$ for $i=1,\dots,m$. Suppose $p(x), g_1(x),\dots, g_m(x)$ are convex SPQ polynomials of degree at most $2d$. Then, the optimal value $p^*$ of~\eqref{eq:pop} can be computed by solving the following semidefinite program:
\begin{equation}
\hspace{3cm}
\begin{aligned}
p^{*} \:\:\: = \:\:\:
\sup\limits_{\gamma \in \R, y \in \R^m} \quad
& \gamma \\
\text{s.t.} \hspace{0.8cm}
& p(x) - \gamma + \sum\limits_{i=1}^{m} y_ig_i(x) \: \in (N^S+N^Q)_{n,2d},\\
& y \geq 0.
\end{aligned}
\label{eq:farkas_applied_sos}
\end{equation}
Moreover, if $p(x)$ is strictly convex, then an optimal solution $x^* \in \R^n$ to \eqref{eq:pop}, i.e., a point $x^* \in K$ such that $p(x^*) = p^*$, can be found by finding zeros of at most $n$ nonnegative univariate polynomials and at most one nonnegative quadratic polynomial, which can be obtained by solving a semidefinite program.
\label{thm:pop_via_sdp}
\end{theorem}

\begin{proof}
By the convex Farkas lemma (see, e.g., \cite{StoWitz}), under the Slater regularity condition, a scalar $\gamma$ is a lower bound on the optimal value $p^*$ of \eqref{eq:pop} if and only if there exists a nonnegative vector $y = (y_1, \dots, y_m)^T$ such that $p(x) - \gamma + \sum_{i=1}^{m} y_ig_i(x)$ is nonnegative. Hence,
\begin{equation}
\hspace{3cm}
\begin{aligned}
p^{*} \:\:\: = \:\:\:
\sup\limits_{\gamma \in \R, y \in \R^m} \quad
& \gamma \\
\text{s.t.} \hspace{0.8cm}
& p(x) - \gamma + \sum\limits_{i=1}^{m} y_ig_i(x) \in N_{n,2d}, \\
& y \geq 0.
\end{aligned}
\label{eq:farkas_applied}
\end{equation}
Since $p(x), g_1(x),\dots, g_m(x)$ are convex SPQ polynomials, $p(x) - \gamma + \sum_{i=1}^{m} y_ig_i(x)$ is a convex SPQ polynomial when $y \geq 0$. By Theorem \ref{thm:convex_nonnegative_sum}, $p(x) - \gamma + \sum_{i=1}^{m} y_ig_i(x)$ is nonnegative if and only if it belongs to the set $(N^S+N^Q)_{n,2d}$. Hence, $p^*$ can be computed by the semidefinite program \eqref{eq:farkas_applied_sos}.\footnote{For the reasons discussed in Section \ref{subsubsec:unconstrained}, the semidefinite program \eqref{eq:farkas_applied_sos} can be much smaller than that of the first level of the Lasserre hierarchy, which would correspond to \eqref{eq:farkas_applied_sos} with $(N^S+N^Q)_{n,2d}$ replaced by $\Sigma_{n,2d}$.}

Next, assume that $p(x)$ is strictly convex. Let $(\gamma^*, y^*)$ be an optimal solution\footnote{The existence of $(\gamma^*, y^*)$ is guaranteed by the convex Farkas lemma.} to~\eqref{eq:farkas_applied_sos} (or equivalently~\eqref{eq:farkas_applied}). We have $\gamma^* = p^*$ and $y^* \geq 0$. Consider the polynomial
$$L(x) \mathrel{\mathop:}= p(x) - \gamma^* + \sum\limits_{i=1}^{m} y_i^*g_i(x).$$
Note that $L(x)$ is a nonnegative, strictly convex SPQ polynomial. Therefore, it has a unique minimizer $\bar x \in \R^n$. By Proposition \ref{prop:unconstrained_recover}, the point $\bar x$ can be obtained by finding zeros of at most $n$ nonnegative univariate polynomials and at most one nonnegative quadratic polynomial.

We now show that $\bar x \in K$ and $p(\bar x) = p^*$, and therefore $\bar x$ is an optimal solution to \eqref{eq:pop}. Observe that $L(\bar x) = 0$ since $L(\bar x) > 0$ would imply that the optimal value of~\eqref{eq:farkas_applied} is larger than $p^*$. We claim that there exists a point $\hat x \in K$ such that $L(\hat x) = 0$. Assume for the sake of contradiction that for every $x$ satisfying $g_i(x) \leq 0$, $i=1,\dots,m$, we have
$$L(x) = p(x) - p^* + \sum\limits_{i=1}^{m} y_i^*g_i(x) > 0.$$
As $y_i^*g_i(x) \leq 0$ for $i=1,\dots,m$, it follows that for every $x \in K$, we have $p^* < p(x)$, a contradiction since we know that $p^* = p(x^*)$ for some $x^* \in K$. This shows that there exists a point $\hat x \in K$ such that $L(\hat x) = 0$. But since $\bar x$ is the unique solution of $L(x) = 0$, we must have $\hat x = \bar x$, and therefore $\bar x \in K$. 

Finally, we show that $p(\bar x) = p^*$. Assume for the sake of contradiction that $p^* < p(\bar x)$. Let $x^*$ be the optimal solution to \eqref{eq:pop}, i.e., the (unique) point satisfying $g_i(x^*) \leq 0$ for $i=1,\dots,m$ and $p(x^*) = p^*$. As $\bar x$ is the unique minimizer of the nonnegative polynomial $L(x)$, we must have $L(x^*) > 0$. Since
$$L(x^*) = p(x^*) - p^* +  \sum_{i=1}^{m} y_i^*g_i(x^*),$$
we conclude that $ \sum_{i=1}^{m} y_i^*g_i(x^*) > 0$. But this contradicts the inequalities $y_i^* \geq 0$ and $g_i(x^*) \leq 0$ for $i=1,\dots,m$. Hence $x^* = \bar x$, and therefore $p(\bar x) = p^*$.
\end{proof}

\section{Applications of SPQ Polynomials} \label{sec:applications}
\reqnomode

In this section, we present three potential applications involving SPQ polynomials.

\subsection{Upper bounds on the sparsity of solutions to linear programs}
\label{subsec:lower_bounds_sparsity}

Given a matrix $A \in \R^{m \times n}$ with $m < n$ and a vector $b \in \R^m$, what is the maximum number of zeros that a vector in the polytope\footnote{By a simple rescaling argument, the results of this section generalize to the case where the polytope is of the form $\hat P \mathrel{\mathop:}= \{y \in \R^n \:|\: \hat Ay = \hat b, -\lambda_i \leq y_i \leq \lambda_i, i=1,\dots,n\}$ for some scalars $\lambda_1, \dots, \lambda_n > 0$.}
$P \mathrel{\mathop:}= \{x \in \R^n \:|\: Ax=b, -1 \leq x \leq 1\}$ can have? 
If we denote the $\ell_0$-pseudonorm of a vector $x \in \R^n$ by $||x||_0 \mathrel{\mathop:}= \big |\{i \:|\: x_i \neq 0\} \big|$, then the answer is $n - p_0$, where $p_0$ is the optimal value of the following problem:
\begin{equation*}
\hspace{1.5cm}
\begin{aligned}
(P_0) \hspace{2.5cm}
p_0 \:\:\: \mathrel{\mathop:}= \:\:\:
\min\limits_{x \in \R^n} \quad
& ||x||_0 \\
\text{s.t.} \hspace{0.5cm}
& Ax = b,\\
& -1 \leq x_i \leq 1, \quad i=1,\dots, n.
\end{aligned}
\label{eq:l0_minimization}
\end{equation*}
Computing $p_0$, however, typically requires an intractable enumerative search, as the optimization problem $(P_0)$ is a nonconvex NP-hard problem \cite{Natar}. While many algorithms have been proposed to provide upper bounds on $p_0$, methods that produce lower bounds are less common. One approach to obtain a lower bound on $p_0$ is to replace the $\ell_0$-pseudonorm by its convex envelope over the hypercube, which is the $\ell_1$-norm:
\begin{equation*}
\hspace{1.5cm}
\begin{aligned}
(P_1) \hspace{2.5cm}
p_1 \:\:\: \mathrel{\mathop:}= \:\:\:
\min\limits_{x \in \R^n} \quad
& ||x||_1 \\
\text{s.t.} \hspace{0.5cm}
& Ax = b,\\
& -1 \leq x_i \leq 1, \quad i=1,\dots, n.
\end{aligned}
\label{eq:l1_minimization}
\end{equation*}
Here, $||x||_1 = \sum_{i=1}^{n} |x_i|$. Unlike $(P_0)$, problem $(P_1)$ can be solved efficiently as it can be recast as a linear program. Because $||x||_1 \leq ||x||_0$ for any $x$ with $-1 \leq x_i \leq 1$, $i=1,\dots,n$, we have $\lceil p_1 \rceil \leq p_0$. Noting that the $\ell_0$-pseudonorm and the $\ell_1$-norm are both separable functions, it is natural to ask whether one can improve the lower bound that $(P_1)$ provides by considering separable polynomials. This is the question we study in this subsection. We propose to replace the objective function in~$(P_0)$ by a separable polynomial $\sum_{i=1}^{n} u_i(x_i)$ to obtain ``input-independent" (Section \ref{subsubsec:independent}) and ``input-dependent" (Section \ref{subsubsec:dependent}) surrogates for the $\ell_0$-pseudonorm. As we shall see, our approach leads to a nonnegativity constraint on SPQ polynomials and results in semidefinite programming-based lower bounds on $p_0$. We call the polynomials $u_1, \dots, u_n$ \emph{penalty polynomials}. In the next two subsections, we discuss the choice of these polynomials.

\subsubsection{Input-independent penalty polynomials}\label{subsubsec:independent}

We first consider the setting where the penalty polynomials $u_1, \dots, u_n$ are chosen to be independent of $A$ and $b$, similar to the $\ell_1$-norm approach in $(P_1)$. Consider the optimization problem
\begin{equation*}
\hspace{1.5cm}
\begin{aligned}
(P_u) \hspace{2.5cm}
p_u \:\:\: \mathrel{\mathop:}= \:\:\:
\min\limits_{x \in \R^n} \quad
& \sum\limits_{i=1}^{n} u_i(x_i) \\
\text{s.t.} \hspace{0.5cm}
& Ax = b,\\
& -1 \leq x_i \leq 1, \quad i=1,\dots, n,
\end{aligned}
\label{eq:our_minimization}
\end{equation*}
where, for $i=1,\dots,n$, $u_i(x_i)$ is a univariate polynomial satisfying $u_i(0) \leq 0$ and $u_i(x_i) \leq 1$ when $x_i \in [-1,1]$. It is easy to see that for any choice of such penalty polynomials, we have $\lceil p_u \rceil \leq p_0$. A lower bound on $p_u$ can be obtained by solving the following semidefinite program:
\begin{equation*}
\hspace{-0.3cm}
\begin{aligned}
(P_{\text{fixed}}) \hspace{0.3cm}
p_\text{{fixed}}^u \mathrel{\mathop:}= \:
\max\limits_{\substack{\gamma \in \R, \: \eta, \tau \in \R^n, \\ \mu_1, \dots, \mu_m}} \quad
& \gamma \\
\hspace{-2.5cm} \text{s.t.} \hspace{0.5cm}
& \sum\limits_{i=1}^{n} u_i(x_i) - \gamma - \sum\limits_{j=1}^{m} \mu_j(x) (a_j^Tx-b_j) - \sum\limits_{i=1}^{n} \eta_i (1-x_i) - \sum\limits_{i=1}^{n} \tau_i (1+x_i) \:\: \text{is sos,}\\
& \eta_i \geq 0, \:\: \tau_i \geq 0, \:\: i=1,\dots,n,\\
& \mu_j(x) \: \text{ is an affine polynomial,} \:\: j=1,\dots, m.
\end{aligned}
\label{eq:dual_our_minimization}
\end{equation*}
Here, for $j=1,\dots, m$, $a_j^T$ denotes the $j^{th}$ row of the matrix $A$. Note that the polynomial in the first constraint of $(P_{\text{fixed}})$ is an SPQ polynomial\footnote{If we replace the constants $\eta_i$, $\tau_i$ in $(P_{\text{fixed}})$ by higher even-degree univariate polynomials $\eta_i(x_i)$, $\tau_i(x_i)$, and impose sos constraints on $\eta_i(x_i)$ and $\tau_i(x_i)$ for $i=1,\dots,n$, the polynomial in the first constraint of $(P_{\text{fixed}})$ would still be an SPQ polynomial and the whole program would still be a semidefinite program. Our choice of the constant multipliers $\eta_i,\tau_i$ is for simplicity and due to the fact that increasing the degrees of these multipliers did not result in better bounds in our experiments.}. It is straightforward to check that $p_\text{{fixed}}^u \leq p_u \leq p_0$. Our goal is to first design appropriate penalty polynomials $u_1, \dots, u_n$ that can be used as a proxy for the $\ell_0$-pseudonorm on all instances of $(P_0)$, and then insert them in $(P_{\text{fixed}})$ to approximate $(P_0)$ better than $(P_1)$. In other words, by appropriately choosing $u_1, \dots, u_n$, we hope that $\lceil p_\text{{fixed}}^u  \rceil$ will be a better lower bound on $p_0$ than $\lceil p_1 \rceil$.

In this input-independent setting, it is natural to pick $u_1, \dots, u_n$ so that they are each as close as possible to the $\ell_0$-pseudonorm in one dimension. A possible approach to achieve this goal would be to let $u_1 = \dots = u_n = u$, where $u$ is determined as follows. For a fixed degree $2d$, we let
$$u(t) = c_0 + c_1t + c_2t^2 + \dots + c_{2d}t^{2d},$$
where the coefficients $c_0, c_1, \dots, c_{2d}$ are an optimal solution to the following problem:
\begin{equation}
\hspace{4.5cm}
\begin{aligned}
\max\limits_{c_0, c_1, \dots, c_{2d}} \hspace{0.5cm}
& \int_{-1}^1 u(t)\;\mathrm{d}t \\
\text{s.t.} \hspace{0.5cm}
& c_{2d} \geq 0 \\
& u(0) \leq 0 \\
& u(t) \leq 1 \quad \text{on} \:\: t \in [-1,1].
\end{aligned}
\label{eq:max_integral}
\end{equation}
The constraint $c_{2d} \geq 0$ ensures that the first constraint of $(P_{\text{fixed}})$ can always be satisfied. The objective function and the first two constraints in~\eqref{eq:max_integral} are linear in the decision variables $c_0, c_1, \dots, c_{2d}$. The last constraint requires a univariate polynomial to be nonnegative over an interval and can be turned into an equivalent sos condition through the following proposition. This implies that problem \eqref{eq:max_integral} can be solved as a semidefinite program.

\begin{proposition}[P\'olya-Szeg\"o, Fekete, Markov-Lukacs; see, e.g., \cite{PovRez} for a proof]
A univariate polynomial $p(x)$ of even degree $2d$ is nonnegative over an interval $[a, b]$, with $a<b$, if and only if it can be written as $p(x) = s(x) + (x-a)(b-x)t(x)$, where $t(x)$ and $s(x)$ are sos polynomials of degree at most $2d - 2$ and $2d$ respectively.
\end{proposition}

Problem \eqref{eq:max_integral} yields different polynomials for different degrees. For instance, the optimal univariate polynomials that the solver \cite{MOSEK} returns for degrees $6$ and $10$ are the following (to two digits of accuracy):
\begin{equation*}
\hspace{3cm}
\begin{aligned}
u(t) &= 5.67 t^2 - 10.11 t^4 + 5.45 t^6, \\
u(t) &= 13.00 t^2 - 61.20 t^4 + 128.42 t^6 - 122.78 t^8 + 43.57 t^{10}.
\end{aligned}
\end{equation*}
As one could anticipate from \eqref{eq:max_integral}, our optimal penalty polynomials end up being nonnegative and even. Figure \ref{fig:indep_opt_penalty_poly} illustrates the optimal penalty polynomials for different degrees, together with the $\ell_0$-pseudonorm and the $\ell_1$-norm in dimension one.
\begin{figure}[h]
\centering
\includegraphics[width=0.39\textwidth]{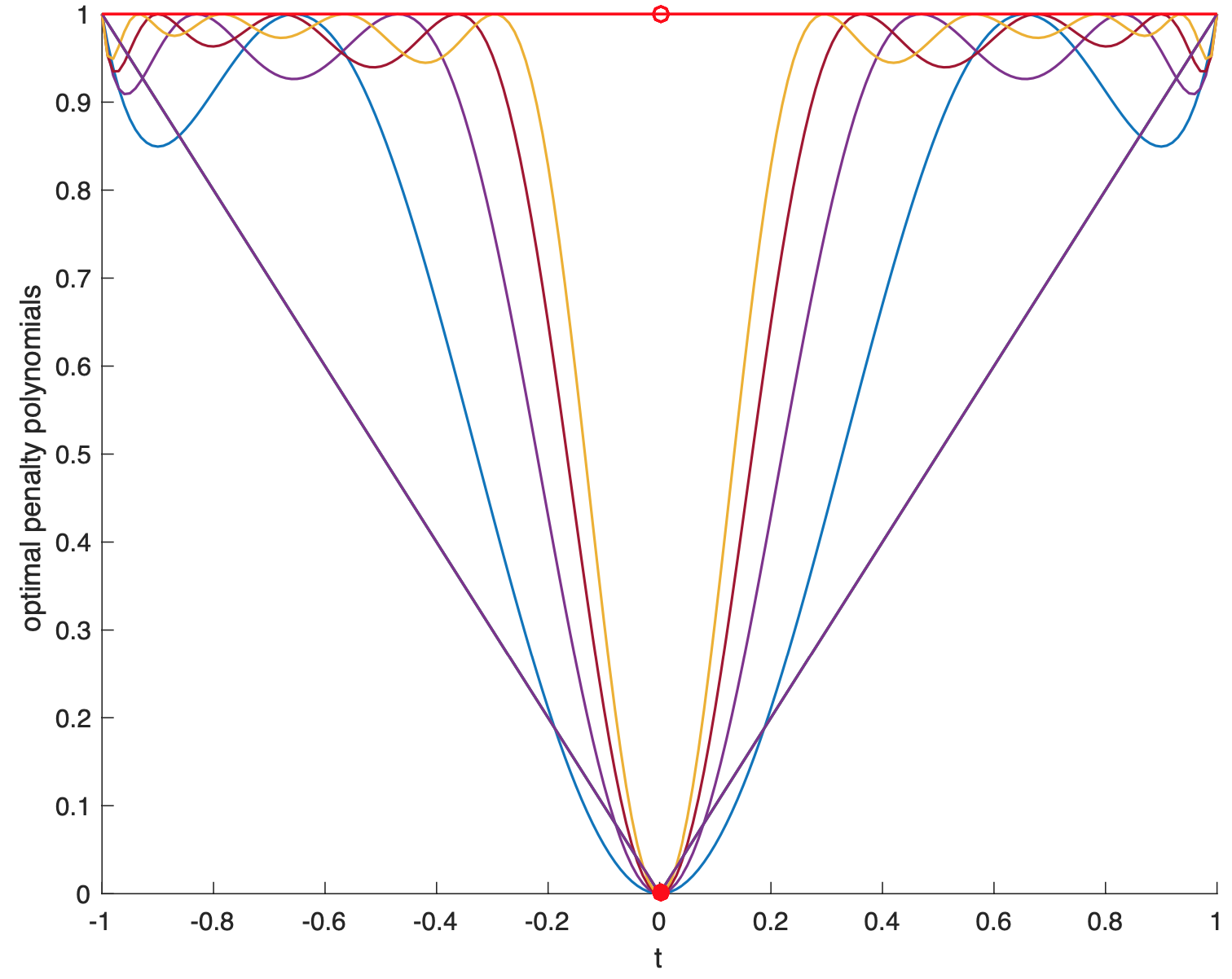}
\vspace{-0.3cm}
\caption{The optimal input-independent penalty polynomials for degrees $2d=6,10,14,18$, together with the $\ell_0$-pseudonorm and the $\ell_1$-norm in dimension one}
\label{fig:indep_opt_penalty_poly}
\end{figure}

\subsubsection{Input-dependent penalty polynomials}\label{subsubsec:dependent}

The penalty polynomials obtained in Section \ref{subsubsec:independent} are input-independent, i.e., they can be used as a proxy for the $\ell_0$-pseudonorm on any instance of $(P_0)$ (given by input $A,b$). In this subsection, we show how the lower bound on $p_0$ can be improved by designing penalty polynomials that take into consideration the problem input. This is achieved by solving the following semidefinite program, where the penalty polynomials $u_1, \dots, u_n$ are part of the decision variables:
\begin{equation*}
\hspace{-1cm}
\begin{aligned}
(P_{\text{free}}) \hspace{0.1cm}
p_{\text{free}}^u \mathrel{\mathop:}=
& \max\limits_{\substack{\gamma \in \R, \: \eta, \tau \in \R^n, \\ u_1, \dots, u_n, \\ r_1, \dots, r_n, s_1 \dots, s_n, \\ \mu_1, \dots, \mu_m, t_{11}, t_{12}, \dots, t_{nm} }}
 \gamma \\
\text{s.t.} \hspace{0.3cm}
& \sum\limits_{i=1}^{n} u_i(x_i) - \gamma - \sum\limits_{j=1}^{m} \mu_j(x) (a_j^Tx-b_j) - \sum\limits_{i=1}^{n} \eta_i (1-x_i) - \sum\limits_{i=1}^{n} \tau_i (1+x_i) \:\: \text{is sos,}\\
& 1 - u_i(x_i) - r_i(x_i) (1-x_i) - s_i(x_i) (1+x_i) - \sum\limits_{j=1}^{m} t_{ij}(x) (a_j^Tx-b_j) \:\: \text{is sos,} \:\: i=1,\dots, n, \\
& u_i(0) \leq 0, \:\: i=1,\dots, n,\\
& \eta_i \geq 0, \:\: \tau_i \geq 0, \:\: i=1,\dots, n, \\
& u_i(x_i) \text{ is a polynomial of degree $2d$,} \:\: i=1,\dots, n, \\
& r_i(x_i), s_i(x_i) \text{ are sos polynomials of degree $2d-2$,} \:\: i=1,\dots, n, \\
& \mu_j(x), \: t_{ij}(x) \: \text{ are affine polynomials,} \:\: i=1,\dots, n, \: j=1,\dots, m.
\end{aligned}
\label{eq:p_free}
\end{equation*}
The second set of constraints\footnote{We remark that the polynomials in the second set of constraints of $(P_{\text{free}})$ are univariate plus quadratic, and therefore, by Theorem \ref{thm:nonneg_uni_plus_quadratic}, are nonnegative if and only if they are sos.} in $(P_{\text{free}})$ implies that for $i=1,\dots,n$, we have $u_i(x_i) \leq 1$ for all $x_i \in [-1,1]$ for which there exist $x_2, \dots, x_n$ with $Ax=b$. This, together with the third set of constraints, ensures that the polynomials $u_1, \dots, u_n$ are underestimators of the $\ell_0$-pseudonorm over the polytope $P = \{x \in \R^n \:|\: Ax=b, -1 \leq x \leq 1\}$. In addition, since the penalty polynomials designed as an optimal solution to \eqref{eq:max_integral} are feasible for $(P_{\text{free}})$, we have $p_{\text{fixed}}^u \leq p_\text{{free}}^u \leq p_0$. We observe in our experiments (see Section~\ref{subsubsec:experiment}) that $p_\text{{free}}^u$ is often a better lower bound on $p_0$ than $p_{\text{fixed}}^u$, and that both $p_{\text{fixed}}^u$ and $p_\text{{free}}^u$ are often better lower bounds on $p_0$ than $p_1$. Figure \ref{fig:compare_dep_opt_penalty_poly} compares optimal input-independent and input-dependent penalty polynomials of degree $2d=6$ for an input $A \in \R^{5 \times 10}, b \in \R^{5}$ to $(P_0)$ for which $\lceil p_{\text{free}}^u \rceil > \lceil p_{\text{fixed}}^u \rceil$. 
\begin{figure}[h]
\centering
\begin{subfigure}[t]{0.38\textwidth}
\includegraphics[valign=t, width=\textwidth]{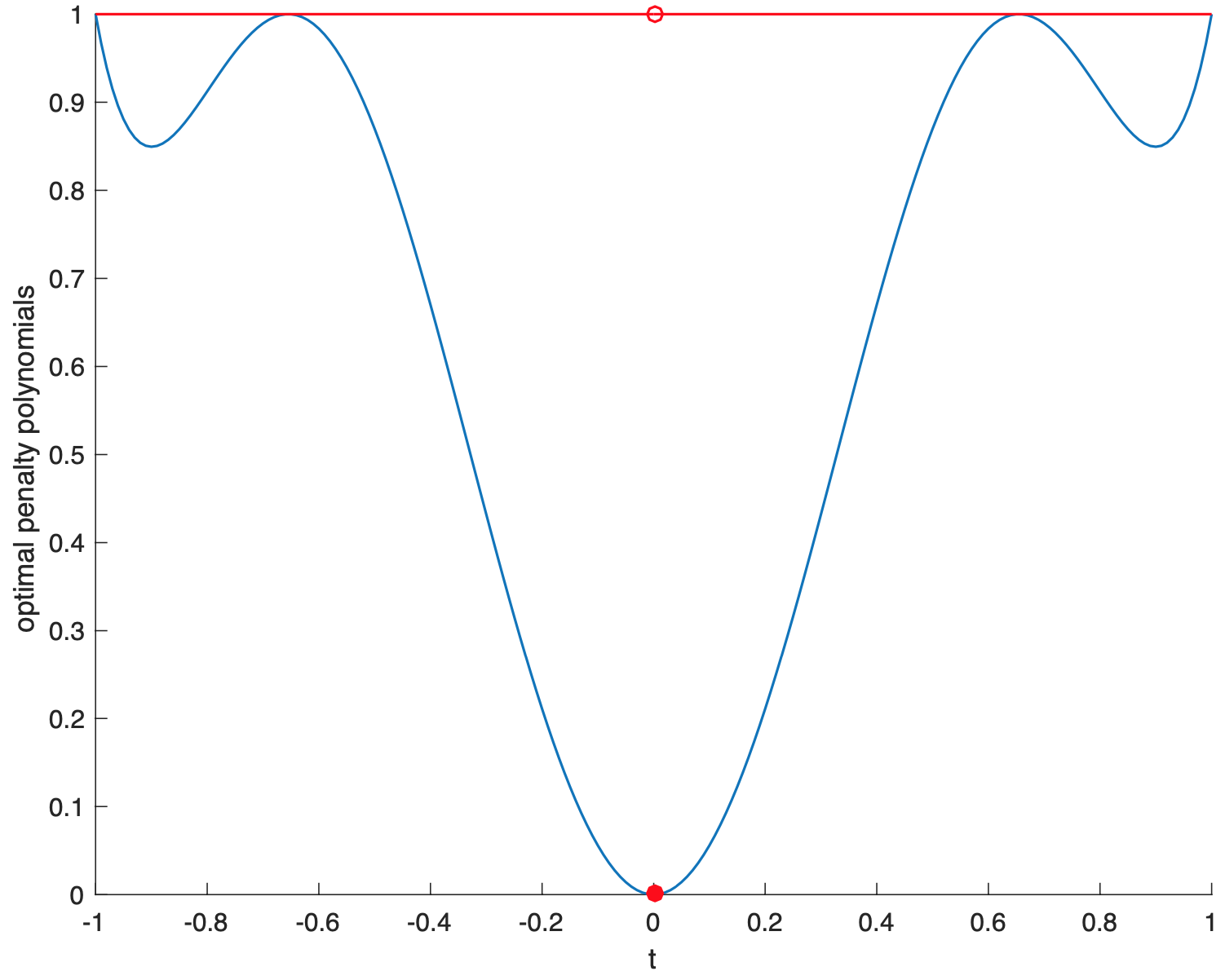}
\vspace{-0.1cm}
\caption{The optimal input-independent penalty polynomial for degree $2d=6$, together with the $\ell_0$-pseudonorm in dimension one}
\label{fig:degree_6_fixed}
\end{subfigure}
\hfill
\begin{subfigure}[t]{0.42\textwidth}
\vspace{-0.13cm}
\includegraphics[valign=t, width=\textwidth]{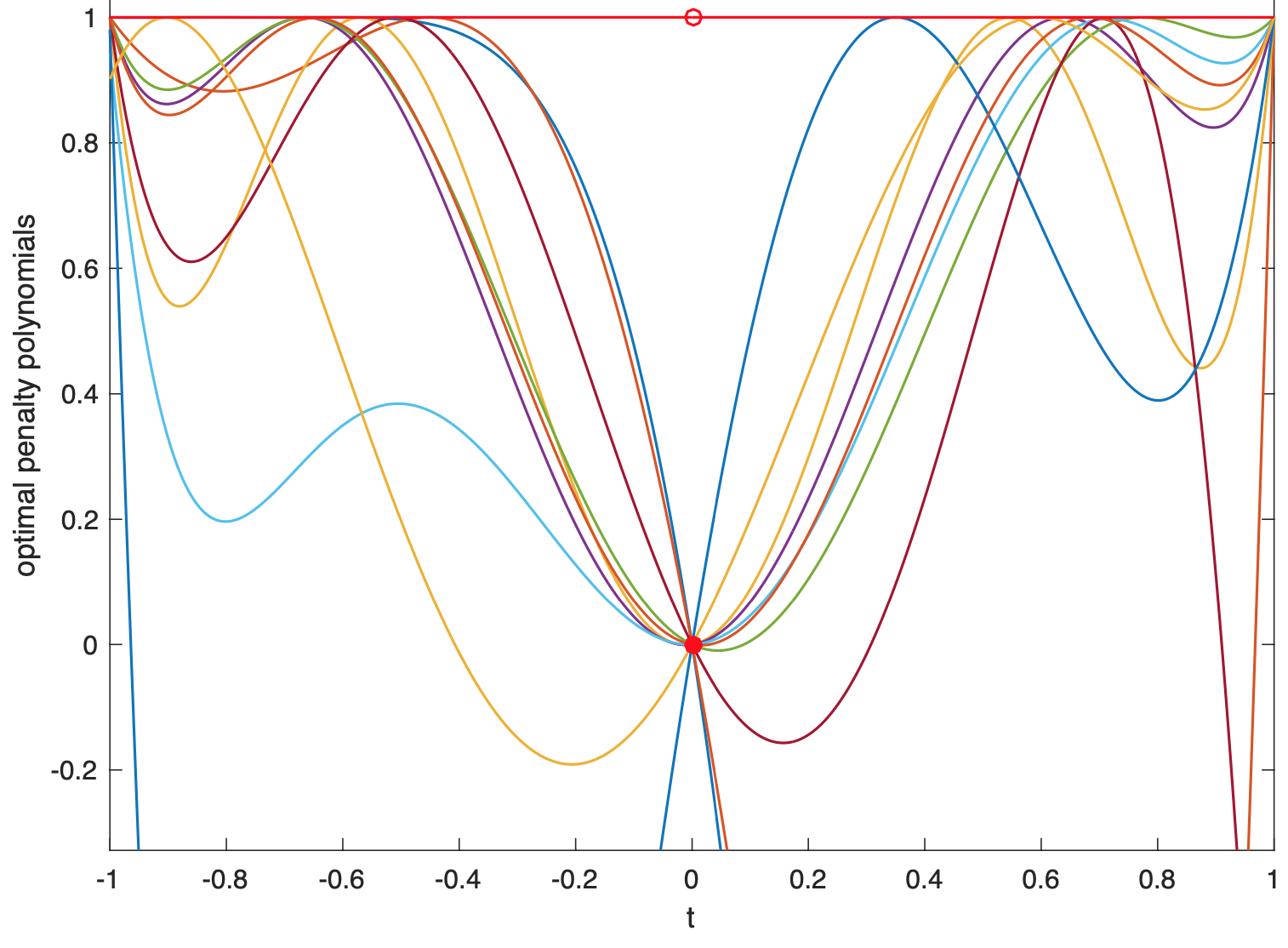}
\vspace{-0.24cm}
\caption{The optimal input-dependent penalty polynomials for degree $2d=6$, together with the $\ell_0$-pseudonorm in dimension one}
\label{fig:dep_opt_penalty_poly}
\end{subfigure}
\vspace{-0.15cm}
\caption{Comparison of optimal input-independent and input-dependent penalty polynomials of degree $2d=6$ for an input $A \in \R^{5 \times 10}, b \in \R^{5}$ to $(P_0)$ for which $\lceil p_{\text{free}}^u \rceil > \lceil p_{\text{fixed}}^u \rceil$}
\label{fig:compare_dep_opt_penalty_poly}
\end{figure}

\subsubsection{Numerical experiments}\label{subsubsec:experiment}

We illustrate our method by a small proof-of-concept numerical experiment. We set $n=10$ and $m = 5$, and generate a matrix $A \in \R^{m \times n}$ and a vector $b \in \R^m$ with entries drawn independently from the standard normal distribution. On $100$ instances where the polytope $P = \{x \in \R^n \:|\: Ax=b, -1 \leq x \leq 1\}$ is nonempty, we compare the lower bounds $p_1$, $p_{\text{fixed}}^u$, $p_\text{{free}}^u$ on $p_0$. The degree of our penalty polynomials is chosen to be $2d = 6$. Table \ref{tab:comparison_lower_bounds} (left) shows the comparison between these three lower bounds and Table \ref{tab:comparison_lower_bounds} (right) compares their ceilings, which are also valid lower bounds on $p_0$. In each row, two lower bounds are compared and the number of times that equalities/strict inequalities hold out of our 100 instances are recorded. The results indicate that replacing the $\ell_1$-norm with appropriate separable penalty polynomials frequently improves the lower bound on $p_0$. For example, it can be seen in Table~\ref{tab:comparison_lower_bounds} (right) that the strict inequality $\lceil p_1 \rceil < \lceil p_{\text{free}}^u \rceil$ holds 99 times out of the 100 instances.
\begin{table}[h]
\centering
\parbox{0.4\linewidth}{
\centering
\begin{tabular}{| m{1cm} || m{0.8cm} | m{0.8cm} | m{0.8cm} || m{1cm} |}
 \hline
  & \centering $>$ & \centering $=$ & \centering $<$ & \tabularnewline
 \hline \hline
  \vspace{0.1cm} \centering $p_1$ & \centering 19 & \centering 0  & \centering 81 & \vspace{0.1cm} \centering $p_{\text{fixed}}^u$ \tabularnewline [0.8ex]
 \hline\hline
\vspace{0.1cm} \centering $p_{\text{fixed}}^u$ & \centering 0 & \centering 0 & \centering 100 & \vspace{0.1cm} \centering $p_\text{{free}}^u$ \tabularnewline [0.8ex]
 \hline\hline
\vspace{0.1cm} \centering $p_1$ & \centering 0 & \centering 0 & \centering 100 & \vspace{0.1cm} \centering $p_\text{{free}}^u$ \tabularnewline [0.8ex]
 \hline
\end{tabular}
}
\hspace{1cm}
\parbox{0.4\linewidth}{
\centering
\begin{tabular}{| m{1.1cm} || m{0.8cm} | m{0.8cm} | m{0.8cm} || m{1.1cm} |}
 \hline
 & \centering $>$ & \centering $=$ & \centering $<$ & \tabularnewline
 \hline  \hline
\vspace{0.1cm} \centering $\lceil p_1 \rceil$ & \centering 0 & \centering 75 & \centering 25 & \vspace{0.1cm} \centering $\lceil p_{\text{fixed}}^u \rceil$ \tabularnewline [0.8ex]
 \hline\hline
\vspace{0.1cm} \centering $\lceil p_{\text{fixed}}^u \rceil$ & \centering 0 & \centering 12 & \centering 88 &  \vspace{0.1cm} \centering $\lceil p_\text{{free}}^u \rceil$ \tabularnewline [0.8ex]
 \hline\hline
\vspace{0.1cm} \centering $\lceil p_1 \rceil$ & \centering 0 & \centering 1 & \centering 99 &\vspace{0.1cm} \centering $\lceil p_\text{{free}}^u \rceil$ \tabularnewline [0.8ex]
 \hline
\end{tabular}
}
\caption{Pairwise comparison of lower bounds $p_1, p_{\text{fixed}}^u, p_\text{{free}}^u$ on $p_0$ (left) and pairwise comparison of their ceilings (right): in each row, the number of times (out of 100 randomly generated instances of $(P_0)$) that equalities/strict inequalities hold between two quantities are recorded}
\label{tab:comparison_lower_bounds}
\end{table}

\vspace{-0.36cm}
\subsection{Convex SPQ polynomial regression}\label{subsec:spq_regression}

In this subsection, we consider the problem of convex polynomial regression (see \cite{CurHall}). We assume that we have $m$ noisy measurements $(w_i, y_i)$, $i = 1, \dots, m$, of an unknown convex function $f: \R^n \rightarrow \R$. To extrapolate this function at new points, we would like to find a polynomial $p: \R^n \rightarrow \R$ (of a given degree) which best explains the observations. This can be done, e.g., by minimizing the sum of the absolute deviations between the observed values and the predicted ones:
\vspace{-0.2cm}
\begin{equation}
\hspace{5.3cm}
\begin{aligned}
& \min\limits_{p}
& & \sum\limits_{i=1}^{m} |p(w_i)-y_i|.
\end{aligned}
\end{equation}
In order to exploit the fact that the underlying function $f$ is convex, one would like to impose a convexity constraint on the regressor $p(x)$. However, the convexity constraint makes the resulting regression problem intractable. In \cite{CurHall}, the authors instead impose an sos-convexity constraint on $p(x)$ and solve the resulting regression problem by semidefinite programming. 

Suppose that in addition to being convex, the underlying unknown function $f$ is known to have low-degree interactions between its variables. In that case, it is natural to approximate $f$ with a convex SPQ polynomial. As we see below, this will also help significantly with the scalability of the resulting regression problem. We thus consider the following SPQ convex regression problem:
\begin{equation}
\hspace{4cm}
\begin{aligned}
\min\limits_{p \text{ SPQ and of deg. } 2d} \hspace{0.85cm}
& \sum\limits_{i=1}^{m} |p(w_i)-y_i|\\
\text{s.t.} \hspace{1.2cm}
& H(x) \succeq 0.
\end{aligned}
\label{eq:convex_spq_regression}
\end{equation}
Here, the decision variables are the coefficients of $p(x)$, the degree $2d$ is fixed, and the matrix $H(x)$ is the Hessian of $p(x)$. Observe that by the proof of Theorem \ref{thm:convex_spq_sos_convex}, Problem \eqref{eq:convex_spq_regression} is equivalent to
\begin{equation}
\hspace{2.4cm}
\begin{aligned}
\min_{\substack{\text{$p$ SPQ and of deg. $2d$},\\ u_1, \dots, u_n \text{ of deg. $2d-2$,}\\ Q \in \R^{n \times n}}} \hspace{0.85cm}
& \sum\limits_{i=1}^{m} |p(w_i)-y_i|\\
\text{s.t.} \hspace{1.2cm}
& H(x) = Q + \text{Diag}(u_1(x_1),\dots,u_n(x_n)) \quad \forall x \in \R^n,\\
& u_i(x_i) \text{ is sos for } i=1,\dots,n, \\
& Q \succeq 0,
\end{aligned}
\label{eq:convex_spq_regression_cheap}
\end{equation}
where $Q$ is a matrix with constant entries, $u_i(x_i)$ is a univariate polynomial in $x_i$ for $i=1,\dots,n$, $\text{Diag}(u_1(x_1),\dots,u_n(x_n))$ is the diagonal matrix with the vector $(u_1(x_1),\dots,u_n(x_n))^T$ on its diagonal, and the equality constraint in \eqref{eq:convex_spq_regression_cheap} is imposed by coefficient matching.

For our numerical experiment, we consider the following family of convex functions:
\begin{equation}
\hspace{3.4cm}
f_{a,b}(x) = \log \Big (\sum\limits_{i=1}^{n} a_i e^{b_ix_i} \Big ) + x^T (aa^T + I)x + b^Tx,
\label{eq:family_convex_func}
\end{equation}
where the entries of $a, b \in \R^n$ are drawn uniformly and independently from the interval $[0, 4]$ and $[-2, 2]$ respectively, and $I$ is the $n \times n$ identity matrix. We consider $r=100$ different instances of functions in $10$ variables thus generated. For each instance, we have a training (resp. test) set of $m=300$ (resp. $t=100$) random vectors $w_i \in \R^{10}$ drawn independently from the (multivariate) standard normal distribution. The values $y_i$ are then computed as $y_i = f_{a,b}(w_i) + \epsilon_i$, where $\epsilon_i$ is again chosen independently from the standard normal distribution. Restricting ourselves to polynomials of degree $2d = 4$, we compare the performances of the four regression problems shown in Table \ref{tab:regression_problems}.
\begin{table}[h]
\centering
\begin{tabular}{|| m{4.48cm} | m{11.6cm} ||} 
\hline
\vspace{0.12cm} Polynomial regression & \vspace{0.1cm}
$\quad \min\limits_{p} \quad \sum\limits_{i=1}^{m} |p(w_i)-y_i|$. \\  [2ex] 
\hline
\vspace{0.12cm} SPQ polynomial regression & \vspace{0.1cm}
$\hspace{0.275cm} \min\limits_{p \text{ SPQ}} \:\:\:\: \sum\limits_{i=1}^{m} |p(w_i)-y_i|$.  \\ [2ex] 
\hline
Sos-convex polynomial regression &
$\quad \min\limits_{p} \quad \sum\limits_{i=1}^{m} |p(w_i)-y_i| \quad \text{s.t.} \:\: H(x)\text{ is an sos-matrix.}$ \\ [2ex] 
\hline
\vspace{0.12cm} SPQ convex polynomial regression & \vspace{0.1cm}
$\hspace{0.28cm} \min\limits_{p \text{ SPQ}} \:\:\:\: \sum\limits_{i=1}^{m} |p(w_i)-y_i| \quad \text{s.t.} \:\: H(x) = Q + \text{Diag}(u_1(x_1),\dots,u_n(x_n)) \quad \forall x,$\\ [2ex] 
& \hspace{4.55cm} $Q \succeq 0, \:\: u_i(x_i) \text{ is sos for } i=1,\dots,n.$
\\ [2ex] 
\hline
\end{tabular}
\caption{The four regression problems considered in our experiments}
\label{tab:regression_problems}
\end{table}

We solve the four regression problems using the training set and obtain an optimal polynomial $p^*$ in each case. We then calculate the average absolute deviation error and the maximum absolute deviation error over the test set:
$$\text{Avg Dev}_{a,b} = \frac{1}{t} \: \sum\limits_{i=1}^{t} \big |f_{a,b} (w_i) - p^*(w_i) \big | \:\: , \hspace{1.2cm} \text{Max Dev}_{a,b} = \max_{1 \leq i \leq t} \big |f_{a,b} (w_i) - p^*(w_i) \big|  \:\: .$$
The histograms in Figure \ref{fig:regression_plots} compare these test errors over $r=100$ instances. As the results illustrate, we are able to obtain significantly smaller errors with SPQ convex polynomial regression.
\begin{figure}[h]
\centering
\begin{minipage}[b]{0.44\textwidth}
\includegraphics[width=\textwidth]{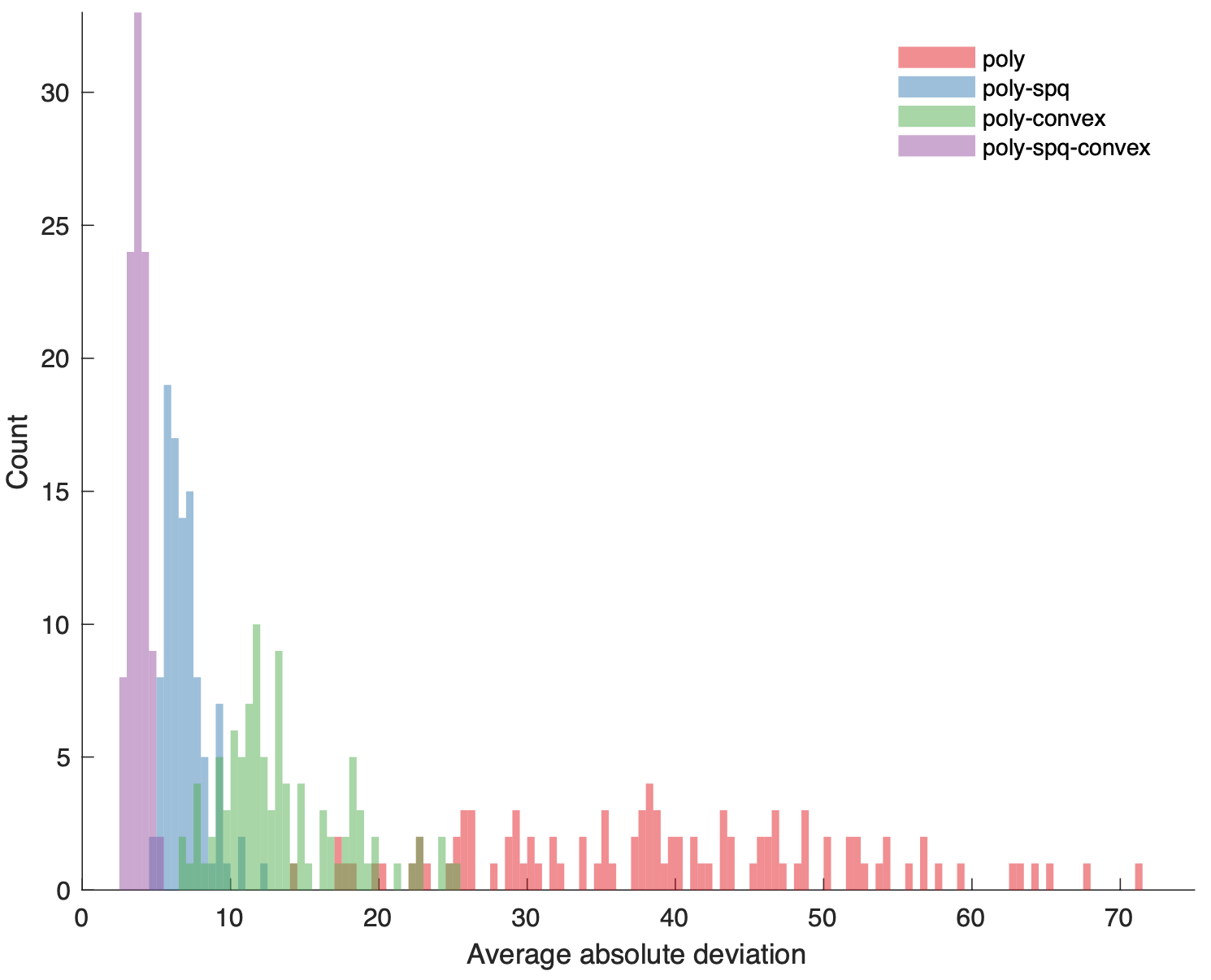}
\end{minipage}
\begin{minipage}[b]{0.44\textwidth}
\includegraphics[width=\textwidth]{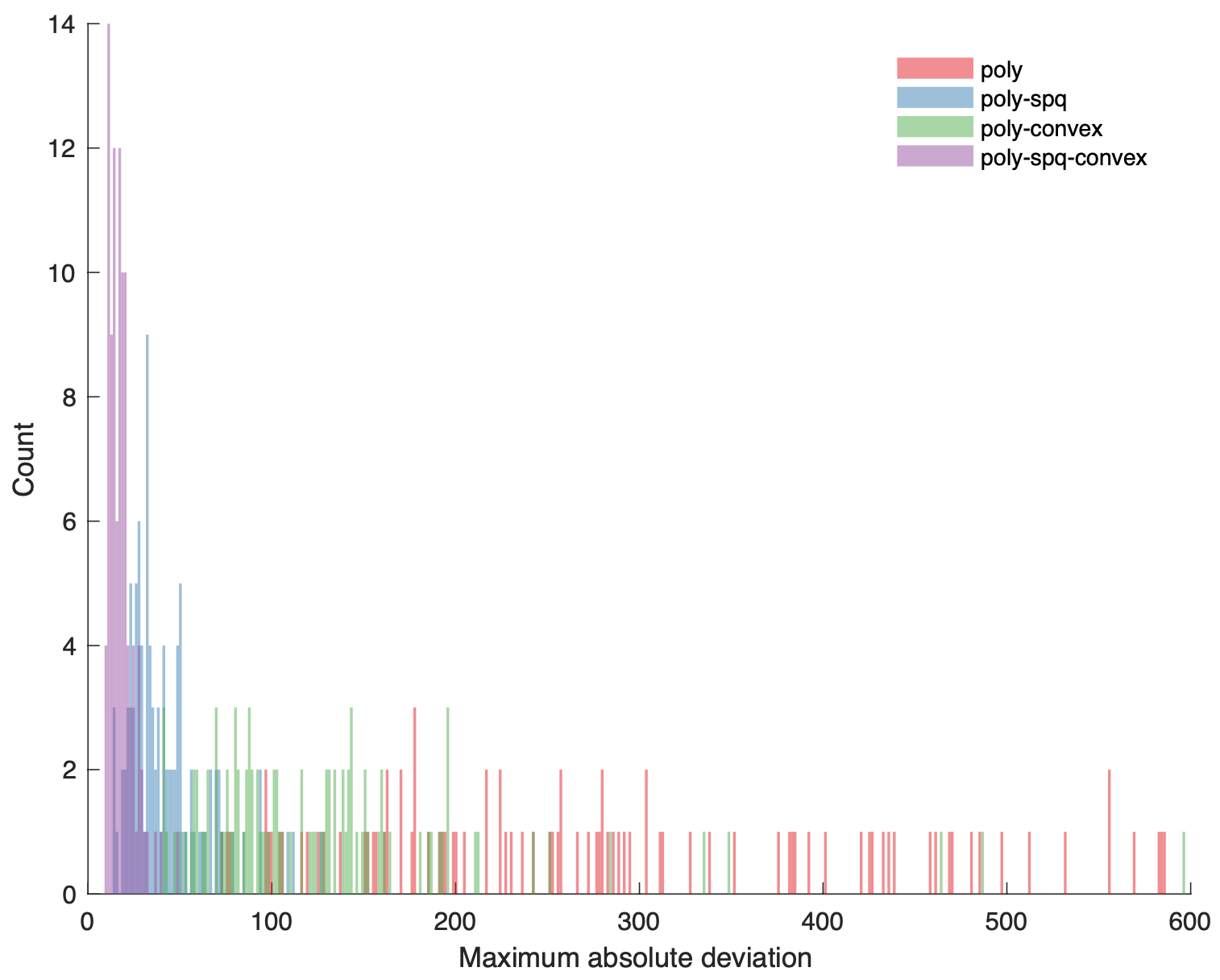}
\end{minipage}
\vspace{-0.1cm}
\caption{Histograms demonstrating test set performance of the different regression approaches over 100 functions randomly chosen from the family of convex functions $f_{a,b}$ given in \eqref{eq:family_convex_func}}
\label{fig:regression_plots}
\end{figure}

\newpage
We shall also remark on the running times of sos-convex and SPQ convex polynomial regression problems. Recall that for SPQ convex polynomial regression, instead of requiring $H(x)$ to be an sos-matrix, we equivalently require $H(x)$ to be the sum of a positive semidefinite constant matrix and a diagonal matrix with entries that are univariate sos polynomials. As can be seen in Table \ref{tab:regression_runtime_comparison}, this leads to a significant difference in running times between the sos-convex and SPQ convex regression formulations already at $n=10$. While we are unable to run sos-convex regression programs beyond $n = 18$ due to memory constraints, SPQ convex regression programs take less than a minute to execute for $n = 80$.
\begin{table}[h]
\centering
\begin{tabular}{|| m{3.46cm} | m{1cm} | m{1cm} | m{1cm} | m{1cm} | m{1.1cm} | m{1cm} | m{1cm} | m{1cm} | m{1cm} ||} 
 \hline
\centering $n$ & \centering $10$ & \centering $12$ & \centering $14$ & \centering $16$ & \centering $18$ & \centering $20$ & \centering $40$ & \centering $60$ & \centering $80$ \tabularnewline [0.5ex]  
 \hline\hline
\centering \small{SPQ convex regression} & \centering 4.6 & \centering 4.7 & \centering 5.9 & \centering 6.0 & \centering 6.4 & \centering 8.8 & \centering 20.1  & \centering 32.8  & \centering  48.8 \tabularnewline  [0.9ex] 
 \hline
\centering \small{Sos-convex regression} & \centering 23.3 & \centering 69.2 & \centering 267.4 & \centering 825.7 & \centering 17778.5 & \centering NA & \centering NA & \centering NA & \centering NA \tabularnewline [0.9ex] 
 \hline
\end{tabular}
\vspace{-0.1cm}
\caption{Comparison of running times (in seconds) averaged over 5 instances for sos-convex and SPQ convex polynomial regression on problems of increasing size with degree $2d = 4$}
\label{tab:regression_runtime_comparison}
\end{table}

\vspace{-0.7cm}

\subsection{The Newton-SPQ method}\label{subsec:newton_spq}

In this subsection, we propose a generalization of Newton's method for minimizing a multivariate function. Recall that Newton's method for minimizing a function $f:~\mathbb{R}^n~\rightarrow~\mathbb{R}$ first approximates $f$ with its second-order Taylor expansion at a current iterate $\hat x \in \R^n$:
\vspace{-0.1cm}
\begin{equation}
\hspace{2.8cm}
q(x) \mathrel{\mathop:}= f(\hat x) + \nabla f(\hat x)^T(x-\hat x) + \frac{1}{2} (x - \hat x)^T \nabla^2 f(\hat x) (x - \hat x).
\label{eq:quad_taylor}
\end{equation}
It then chooses the next iterate to be a critical point of the quadratic polynomial $q(x)$. If the function $f$ is convex, this critical point will be a global minimum of $q(x)$. When $f$ is not convex, what is commonly done in practice is to add a scaled identity matrix to the Hessian of $f$, with the scale adjusted at every iteration so that the resulting matrix is positive semidefinite.

Here, we propose to approximate $f$ with an SPQ polynomial around the current iterate and minimize the resulting SPQ polynomial in order to obtain the next iterate. The idea of approximating $f$ with an SPQ polynomial instead of a quadratic polynomial allows us to take advantage of higher-order information, while maintaining tractability of the iterative method through sos techniques and a very structured semidefinite program.

Let $f(x) = f(x_1, \dots, x_n)$ be the function that we would like to minimize. At the current iterate $\hat x = (\hat x_1, \dots, \hat x_n)$, we find the best SPQ polynomial approximation of $f$ as follows. We first keep the quadratic approximation \eqref{eq:quad_taylor} that would also appear in Newton's method. Then, in order to capture higher-order information, we add, for each variable $x_i$, a higher-order Taylor expansion of the univariate function $f_i(x_i)$ obtained by restricting $f$ to the line $\hat x + \alpha e_i$, where $e_i$ is the $i^{\text{th}}$ standard coordinate vector. We need to also subtract the quadratic part of these univariate Taylor expansions as they are already accounted for in the quadratic approximation of $f$. More specifically, let $q(x)$ be the second-order Taylor expansion of $f$ at $\hat x$ as given in \eqref{eq:quad_taylor}. For each variable $x_i$, let
\vspace{-0.06cm}
$$f_i(x_i) \mathrel{\mathop:}= f(\hat x_1, \dots, \hat x_{i-1}, x_i, \hat x_{i+1}, \dots, \hat x_n)$$ 
be the univariate function that is obtained from $f$ by setting $x_j = \hat x_j$ for $j \neq i$. For a fixed integer $d \geq 2$, let $u_{i, 2}(x_i)$ and $u_{i, 2d}(x_i)$ be the second-order and $2d^{\text{th}}$-order Taylor expansions of $f_i(x_i)$ at $\hat x_i$, respectively. We then obtain the following degree-$2d$ SPQ polynomial approximation of $f$ at $\hat x$:
\vspace{-0.06cm}
$$p(x) \mathrel{\mathop:}= q(x) + \mathlarger{\sum\limits_{i=1}^{n}} \Big (u_{i, 2d}(x_i) - u_{i, 2}(x_i) \Big).$$
Similar to the modification done in Newton's method to make $q(x)$ convex, we propose to modify $p(x)$ by adding $\lambda \sum_{i=1}^n (x_i^{2d} + x_i^2)$, where $\lambda$ is the smallest nonnegative scalar that makes the polynomial $p(x) + \lambda \sum_{i=1}^n (x_i^{2d} + x_i^2)$ convex. It is not hard to see that such $\lambda$ always exists (see \cite{AhmHall2}), and by Theorem~\ref{thm:convex_spq_sos_convex}, $\lambda$ can be computed by a semidefinite program where the size of the largest semidefinite constraint is $\max \{n \times n, \: d \times d \}$. To obtain the next iterate, we minimize the convex SPQ polynomial $p(x)~+~\lambda \sum_{i=1}^n (x_i^{2d} + x_i^2)$ and recover a point at which the optimal value is achieved. This minimization problem is precisely of the type studied in Section \ref{subsubsec:unconstrained}. Therefore, by using Theorem \ref{thm:convex_nonnegative_sum}, the minimization can be carried out by a scalable semidefinite program, and an optimal solution can be recovered, e.g., using Proposition \ref{prop:unconstrained_recover}.

As an example, we compare the Newton's method and the Newton-SPQ method to minimize the bivariate function
\vspace{-0.3cm}
\begin{equation}
\hspace{2.4cm}
f(x_1, x_2) = 2(x_1-x_2) \: \arctan(x_1-x_2) - \log(1+(x_1-x_2)^2) + x_1^2,
\label{eq:f_newton}
\end{equation}
which is strictly convex and has a unique global minimum at $(x_1,x_2) = (0,0)$. We observe that Newton's method fails to converge to the global minimum of $f$ if the coordinates of the initial point $(\hat x_1, \hat x_2)$ are not close to each other. More precisely, Newton's method converges to $(0,0)$ only for initial values $(\hat x_1, \hat x_2)$ that approximately satisfy $|\hat x_1 - \hat x_2| \leq 1.3917$. By contrast, we observe that the Newton-SPQ method with $2d=4$ converges to the global minimum of $f$ for every initial point. The plots in Figure~\ref{fig:test} demonstrate this situation where the basins of attraction of the global minimum for both methods are shown with initial points $(\hat x_1, \hat x_2)$ chosen from the domain $[-20, 20] \times [-20, 20]$.
\begin{figure}[h]
\centering
\begin{minipage}[b]{0.29\textwidth}
\includegraphics[width=\textwidth]{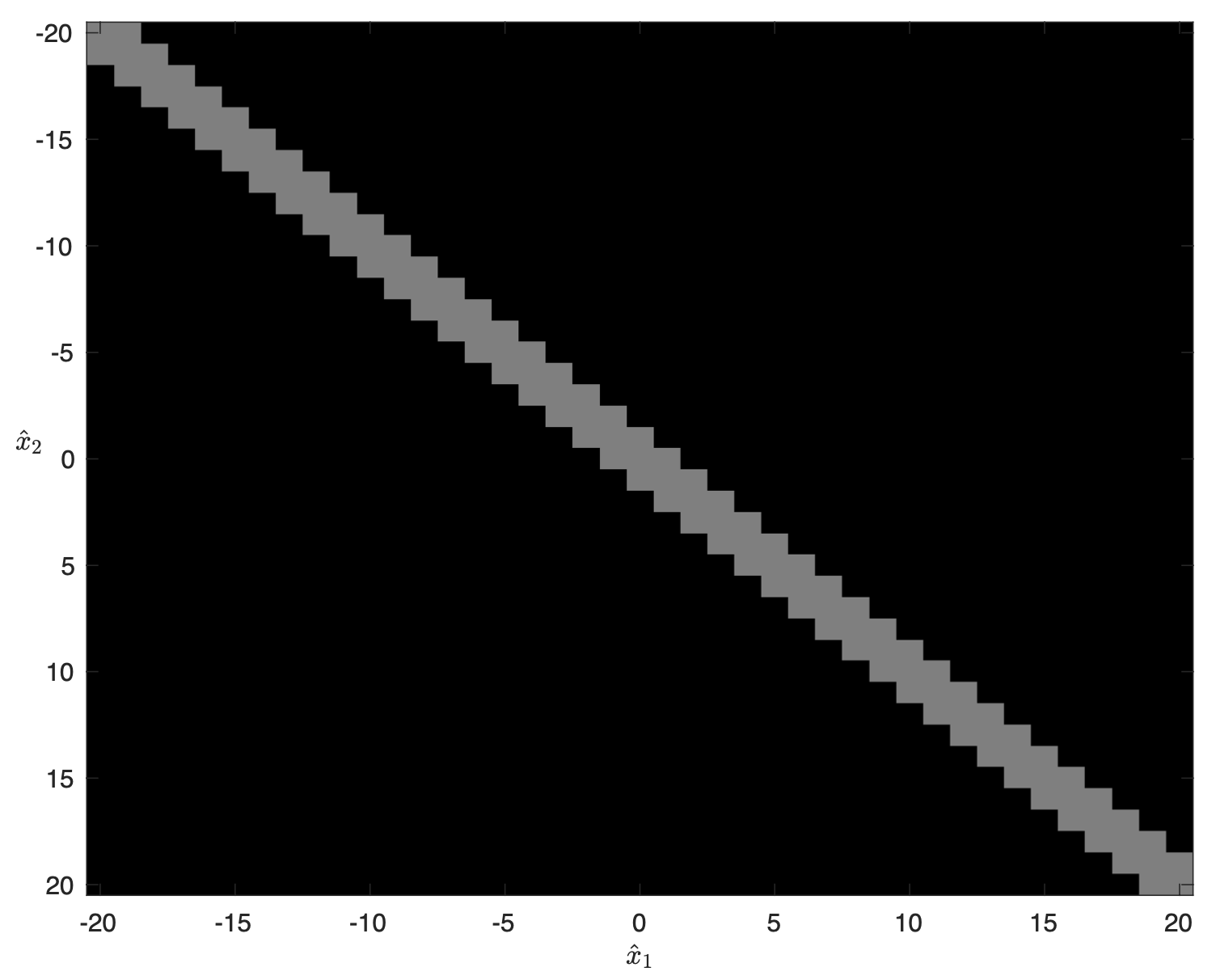}
\end{minipage}
\hspace{1.5cm}
\begin{minipage}[b]{0.29\textwidth}
\includegraphics[width=\textwidth]{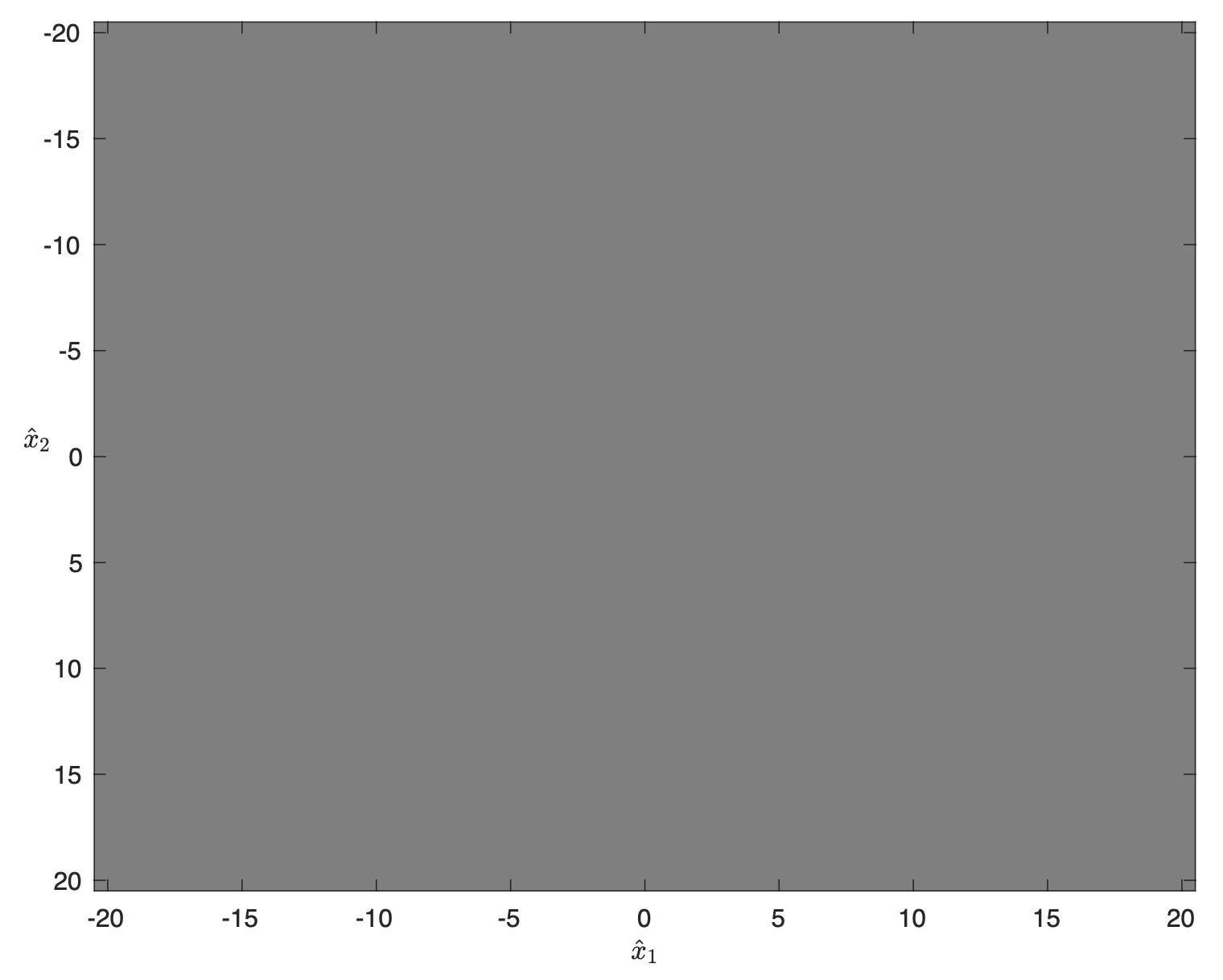}
\end{minipage}
\vspace{-0.17cm}
\caption{Basins of attraction (in light gray) of the global minimum of $f$ in \eqref{eq:f_newton} for the Newton's method (left) and for the Newton-SPQ method (right)}
\label{fig:test}
\end{figure}


\end{document}